\documentclass[11pt,a4paper]{amsart}
\usepackage{amssymb}
\usepackage{amsmath}
\usepackage{amsfonts}
\usepackage[english]{babel}
\usepackage[T1]{fontenc}
\usepackage[latin1]{inputenc}
\usepackage{amsthm}
\usepackage[all]{xy}
\usepackage{a4wide}

\usepackage{dsfont}
\usepackage{verbatim}
\usepackage{graphicx}
\usepackage{tikz,tikz-cd}

\usepackage{mathrsfs}

\theoremstyle{definition}
\newtheorem{thm}{Theorem}[section]
\newtheorem{prop}[thm]{Proposition}
\newtheorem{lma}[thm]{Lemma}
\newtheorem{cor}[thm]{Corollary}

\newtheorem{qu}[thm]{Question}
\newtheorem{qus}[thm]{Questions}

\newtheorem{remark}[thm]{Remark}
\newtheorem{remarks}[thm]{Remarks}
\newtheorem{ex}[thm]{Example}
\newtheorem{fact}[thm]{Fact}
\newtheorem{facts}[thm]{Facts}
\theoremstyle{plain}
\usepackage{verbatim}
\usepackage{graphicx}
\usepackage{tikz-cd}
\usepackage[numbers]{natbib}
\usepackage[shortlabels]{enumitem}
\numberwithin{equation}{section}

\begin{document}

\title{Closed ideals in the algebra of compact-by-approximable operators}
%\newpageafter{title}
\author[Hans-Olav Tylli and Henrik Wirzenius]{Hans-Olav Tylli and Henrik Wirzenius}
\address{Tylli: Department of Mathematics and Statistics, Box 68,
FI-00014 University of Helsinki, Finland}
\email{hans-olav.tylli@helsinki.fi}
\address{Wirzenius: Department of Mathematics and Statistics, Box 68,
FI-00014 University of Helsinki, Finland}
\email{henrik.wirzenius@helsinki.fi}


\begin{abstract}
We construct various examples of non-trivial closed ideals of the compact-by-approximable 
algebra $\mathfrak{A}_X =:\mathcal K(X)/\mathcal A(X)$ 
on Banach spaces $X$ failing the approximation property. 
The examples include the following: (i) if $X$ has cotype $2$, $Y$ has type $2$,  
$\mathfrak{A}_X \neq \{0\}$ and $\mathfrak{A}_Y \neq \{0\}$, then $\mathfrak{A}_{X \oplus Y}$  
has at least $2$ closed ideals,
(ii) there are closed subspaces $X \subset \ell^p$ for $4 < p < \infty$ and $X \subset c_0$
such that $\mathfrak{A}_X$ contains a non-trivial closed ideal, 
(iii) there is a Banach space $Z$ such that 
$\mathfrak{A}_Z$ contains an uncountable lattice of closed ideal having the reverse 
order structure of the
power set of the natural numbers. Some of our examples involve
non-classical approximation properties 
associated to various Banach operator ideals. We also discuss
the existence of compact non-approximable operators $X \to Y$, 
where $X \subset \ell^p$ and $Y \subset \ell^q$
are closed subspaces for $p \neq q$.
\end{abstract}

\maketitle

\section{Introduction}\label{intro}

Given Banach spaces  $X$ and $Y$, let $\mathcal K(X,Y)$ be the class of compact operators 
$X \to Y$. The uniform norm closure 
$\mathcal A(X,Y) = \overline{\mathcal F(X,Y)}$ defines the  
class of approximable operators,
where $\mathcal F(X,Y)$ is the linear subspace consisting 
of the bounded finite-rank operators $X \to Y$. We abbreviate 
$\mathcal K(X) = \mathcal K(X,X)$ and $\mathcal A(X) = \mathcal A(X,X)$ for $X = Y$. 
One obtains the quotient algebra $\mathfrak{A}_X =:  \mathcal K(X)/\mathcal A(X)$
of the  compact-by-approximable operators on $X$, since $\mathcal A(X)$ is a closed two-sided ideal of $\mathcal K(X)$. The quotient  $\mathfrak{A}_X $ is a non-unital Banach algebra 
equipped with the quotient  norm
\[
\Vert S + \mathcal A(X) \Vert = dist(S,\mathcal A(X)), \quad S \in \mathcal K(X).
\]

If $X$ has the approximation property, then it  is well known that  $\mathfrak{A}_X = \{0\}$.
Recall that Banach spaces $X$ failing the approximation property 
are complicated objects to construct or recognise. 
As a consequence the class of compact-by-approximable algebras $\mathfrak{A}_X$  is 
quite intractable, and it has largely  been neglected in the literature. 
Nevertheless, these quotient algebras are natural examples of (typically) non-commutative radical  
Banach algebras, that is, the quotient elements $S + \mathcal A(X)$ are quasi-nilpotent
for all $S \in \mathcal K(X)$.
Recently various facts and problems about such algebras 
were highlighted by Dales \cite{D13}, and his questions motivated the results and examples
 in \cite{TW} about the size of the quotient algebras $\mathfrak{A}_X$
 for classes of Banach spaces $X$.
In this paper we complement and expand our earlier 
study by looking more carefully at the algebraic structure of $\mathfrak{A}_X$, in particular
at examples of non-trivial closed two-sided  ideals. 

\smallskip

If  $X$  fails to have the approximation property and 
$\mathfrak{A}_X \neq \{0\}$, then it is a serious challenge  to construct non-trivial 
closed ideals of $\mathfrak{A}_X$.
In Section \ref{closedids} we provide the first examples of this kind. 
In  Theorem \ref{ideals} we exhibit a class of 
direct sums $X \oplus Y$, where  $\mathfrak{A}_{X\oplus Y}$ has at least two non-trivial, incomparable 
closed ideals.
As a by-product of our discussion we also point out (Corollary \ref{universal})
that the class of non-approximable operators does not contain a universal operator.
In Theorem \ref{nil} we find closed subspaces 
$X \subset \ell^p$ for $p \in [1,\infty)$ and $p \neq 2$, as well as $X \subset c_0$,
for which the quotient algebra $\mathfrak{A}_X$ is non-nilpotent and infinite-dimensional.
This result improves on \cite[section 2]{TW},
and it will also be crucial in some of our later examples 
of closed ideals of the compact-by-approximable algebra.

\smallskip

In Section \ref{compactnona} we systematically discuss the following natural problem: for which 
parameters $p \neq  q$ is it possible to find  closed subspaces $X \subset \ell^p$ and 
$Y \subset \ell^q$, such that
\[
\mathcal A(X,Y) \varsubsetneq \mathcal K(X,Y) \ ?
\]
This is not always the case 
(see Theorem \ref{KMJ}), and our discussion is motivated by  Theorem \ref{ideals}.
The cases  $p,q \in [1,2)$ 
turn out to be related to earlier factorisation results of 
Figiel \cite{F}, Alexander \cite{A} and Bachelis \cite{B} for compact operators. For
$p, q \in (2,\infty)$ we will use non-classical approximation properties with respect to certain Banach operator ideals $\mathcal I$ 
which are contained in the ideal $\mathcal K$ of the compact operators. 
The result (Theorem \ref{reinov}) 
re-examines a sophisticated example of Reinov \cite{Rei82}
about  the failure of duality for $p$-nuclear operators.
The concepts and results in Section \ref{compactnona} enable us to revisit the 
setting of Theorem \ref{ideals}, 
and to exhibit in Example \ref{8ideal} such direct sums $X \oplus Y$,
for which $\mathfrak{A}_{X\oplus Y}$ has (at least) 8  non-trivial closed ideals.
Moreover, in Examples \ref{newideal} and \ref{c0v2} 
we uncover closed subspaces $X \subset \ell^p$ 
for $p \in (4,\infty)$ and $X \subset c_0$, where $\mathfrak{A}_X$
contains a non-trivial closed ideal.

\smallskip

In  Section \ref{quest} we find spaces $X$
for which $\mathcal{CA}(X) \cap \mathcal K(X)$
determines a non-trivial  closed ideal of $\mathfrak{A}_X$, where  
$\mathcal{CA}$ is the class of compactly approximable operators. 
Finally, given any Banach space $X$ such that $X$ has the 
approximation property but $X^*$ fails this property, 
we construct in Example \ref{infty} an associated space $Z$ for which
$\mathfrak{A}_Z$ carries an uncountable family of non-trivial closed ideals
having an explicit order structure. 
We also draw attention to some problems raised by our results. 

\medskip

\textit{Preliminaries.} We briefly recall some  standard concepts that 
will freely be used later. The Banach space $X$ has the approximation property (A.P.)
 if for all compact subsets 
$K \subset X$ and $\varepsilon > 0$ there is a bounded finite-rank operator 
$U \in \mathcal F(X)$  such that 
\begin{equation}\label{ap}
\sup_{x \in K} \Vert x - Ux \Vert < \varepsilon.
\end{equation}
If there is a uniform bound $C < \infty$ such that the approximating operator
$U \in \mathcal F(X)$ in \eqref{ap} can be chosen to satisfy $\Vert U\Vert \le C$, then 
$X$ has the bounded approximation property (B.A.P.). If compact operators 
$U \in \mathcal K(X)$
are allowed in condition \eqref{ap}, then one obtains analogously 
the compact approximation property (C.A.P.) 
and its bounded version B.C.A.P.  For a comprehensive discussion of the
classical approximation properties we refer to
 \cite[1.e and 2.d]{LT1}, \cite[1.g]{LT2}, 
and the survey \cite{C}.
In general,   \cite{AK}, \cite{DJT}  and \cite{LT1} are 
references for unexplained concepts and results related to Banach spaces. 

\smallskip

Let  $\mathcal L(X,Y)$ be the space of bounded linear operators $X \to Y$
for Banach spaces $X$ and $Y$.
We say here that $(\mathcal I, \vert \cdot\vert_\mathcal I)$ is a Banach operator ideal
if $(\mathcal I, \vert \cdot\vert_\mathcal I)$ is a complete normed 
operator ideal in the sense of Pietsch  \cite{Pie}. 
More precisely, this entails that the following conditions 
hold for all Banach spaces $X$ and $Y$:

\begin{itemize}
\item[(BOI1)]
the ideal component $\mathcal I(X,Y)$ is a linear subspace of $\mathcal L(X,Y)$,
and $\vert \cdot\vert_\mathcal I$ is a complete norm in $\mathcal I(X,Y)$
such that $\Vert S\Vert \le \vert S\vert_\mathcal I$ for all $S \in \mathcal I(X,Y)$,

\smallskip

\item[(BOI2)]
the bounded finite-rank operators $\mathcal F(X,Y) \subset \mathcal I(X,Y)$ and 
$\vert x^* \otimes y \vert_\mathcal I = \Vert x^*\Vert \cdot \Vert y\Vert$ for all 
 $x^* \in X^*$ and  $y \in Y$,
 
 \smallskip
 
 \item[(BOI3)] 
 for all Banach spaces $Z$ and $W$ the product  $BSA \in \mathcal I(Z,W)$ for all 
 $S \in \mathcal I(X,Y)$ and all bounded operators $A \in \mathcal L(Z,X)$
 and $B \in \mathcal L(Y,W)$, and in addition
 \begin{equation}\label{id}
 \vert BSA \vert_\mathcal I \le \Vert B\Vert \cdot  \Vert A\Vert  \cdot \vert S \vert_\mathcal I.
 \end{equation}
\end{itemize}

 \smallskip

Above $x^* \otimes y \in \mathcal F(X,Y)$ denotes the operator $x \mapsto x^*(x)y$.  
Banach operator ideals of particular interest in Section \ref{compactnona}
will be the ideal $\mathcal{KS}_r$ of the operators that factor
 compactly through a closed subspace of $\ell^r$, 
 the ideal $\mathcal{SK}_r$ of the (Sinha-Karn) $r$-compact operators, and 
the ideal $\mathcal{QN}_r$ of the quasi $r$-nuclear operators.
We refer to  \cite{Pie}, \cite{DF} and \cite{DJT} as general sources for
classical examples of  operator ideals and for related unexplained concepts 
and constructions.

 \smallskip

Recall that a complex Banach algebra $\mathcal A$ 
is a radical algebra if the spectrum $\sigma(x) = \{0\}$ for all $x \in \mathcal A$, 
where the spectrum is 
computed in the unitisation $\mathcal A^{\#}$ of $\mathcal A$. 
It is known \cite[Theorem 2.5.8(iv)]{D00} that for complex Banach spaces $X$ 
the quotient algebra $\mathfrak{A}_X$ is radical. 
For  real Banach spaces $X$ our interpretation is that
the real quotient algebra $\mathfrak{A}_X$ is radical in the sense that
the real spectrum 
\begin{equation}\label{rrad}
\sigma_{\mathbb R}(S + \mathcal A(X)) =:
\{\lambda \in \mathbb R: \lambda 1 - (S + \mathcal A(X))
\textrm{ is invertible in } (\mathfrak{A}_X)^{\#}\} = \{0\}
\end{equation}
for all $S \in \mathcal K(X)$.
The fact that this holds in the real case is a consequence of  classical 
Riesz-Fredholm theory, which does not depend of the scalar field (see
Proposition \ref{real} for the precise details).
With this understanding
our results and examples are independent of the scalar field 
$\mathbb R$ or $\mathbb C$ of the underlying Banach space $X$.
Recall also from  \cite[Question 2.2.A, page 182]{D00} that it is a longstanding open problem whether there are topologically simple
radical Banach algebras $\mathcal A$, 
that is, $\mathcal A$ has a non-trivial product and no non-trivial closed ideals.
  
\section{Non-trivial closed ideals of $\mathfrak{A}_Z$ and algebraic properties}\label{closedids}

In  this section we construct the first examples of 
non-trivial closed two-sided ideals
in the compact-by-approximable quotient algebra $\mathfrak{A}_Z$ for certain classes of 
Banach spaces $Z$.
We also exhibit closed subspaces $X \subset \ell^p$, 
where $1 \le p < \infty$ and $p \neq 2$, such that the algebra $\mathfrak{A}_X$ is non-nilpotent. 
This result improves \cite[section 2]{TW}, and it will be essential
for some of our later examples of closed ideals. 

\smallskip

Let $Z$ be a Banach space. In view of  Proposition \ref{quotient} below 
we will equivalently exhibit  operator norm closed two-sided (algebraic) 
ideals $\mathcal J$ of $\mathcal K(Z)$,
such that $\mathcal A(Z) \varsubsetneq \mathcal J \varsubsetneq \mathcal K(Z)$. 
Some of the ideals will be defined by internal conditions  
for particular Banach spaces $Z$, in which case the task is to verify that 
the class really is a non-trivial closed ideal of $\mathcal K(Z)$. 
 Banach operator ideals $(\mathcal I, \vert \cdot\vert_\mathcal I)$ 
provide another important source of examples. 
Here $\mathcal I(Z)  =: \mathcal I(Z,Z)$ is a two-sided (algebraic) ideal of $\mathcal L(Z)$, 
but  $\mathcal I(Z)$ is typically not closed in the uniform operator norm. 
 We will reserve the notation $\overline{\mathcal I(X,Y)}$
 for the operator norm closure of $\mathcal I(X,Y)$ in $\mathcal L(X,Y)$ 
 for Banach spaces $X$ and $Y$.
Evidently $\overline{\mathcal I(Z)}$ is a closed (algebraic) ideal of $\mathcal L(Z)$
for every space $Z$.
We emphasize that $\mathcal A(Z) \subset \mathcal I$ 
whenever $\{0\} \neq \mathcal I \subset \mathcal K(Z)$
is a closed (two-sided algebraic) ideal of $\mathcal K(Z)$.
In fact, let
$U \in \mathcal I$ be a non-zero operator, and pick $x \in Z$ and $x^* \in Z^*$ 
such that $\langle Ux,x^* \rangle = 1$.
If $y \in Z$ and $y^* \in Z^*$ are arbitrary, then
\[
(x^* \otimes y) \circ (y^* \otimes Ux) = \langle Ux,x^* \rangle (y^* \otimes y) = 
y^* \otimes y \in \mathcal I,
\]
since $y^* \otimes Ux = U \circ (y^* \otimes x) \in \mathcal I$. 
This implies that $\mathcal A(Z) \subset \mathcal I$.

\smallskip

Let $Z$ be an arbitrary  Banach space and $q: \mathcal K(Z) \to \mathfrak{A}_Z$ 
be the quotient map.  
Our starting point is the easy fact that
there is a bijective correspondence
between the closed non-trivial ideals of $\mathfrak{A}_Z$ and 
the closed non-trivial ideals  of $\mathcal K(Z)$. 
Clearly this is a special instance of a general fact for quotient algebras
$\mathcal A/\mathcal J$, where 
$\mathcal J$ is a closed ideal of the Banach algebra $\mathcal A$.

\begin{prop}\label{quotient}
(i) If $\mathcal A(Z) \varsubsetneq \mathcal I \varsubsetneq \mathcal K(Z)$ is a 
non-trivial closed ideal  of $\mathcal K(Z)$,  then
\[
\{0\} \neq q(\mathcal I) \varsubsetneq  \mathfrak{A}_Z
\]
is a non-trivial closed ideal of  $\mathfrak{A}_Z$. Moreover, if $\mathcal  I_1 \neq \mathcal I_2$, 
then $q(\mathcal  I_1) \neq q(\mathcal  I_2)$.

\smallskip

(ii) If  $\{0\} \neq \mathcal J \varsubsetneq \mathfrak{A}_Z$
is a closed ideal, then 
$\mathcal A(Z) \varsubsetneq q^{-1}(\mathcal J) \varsubsetneq \mathcal K(Z)$ is a 
non-trivial closed ideal of $\mathcal K(Z)$. Moreover,
if $\mathcal  J_1 \neq \mathcal J_2$, then 
$q^{-1}(\mathcal J_1) \neq q^{-1}(\mathcal J_2)$.

\end{prop}

\begin{proof}
(i) Note  that $q(\mathcal I)$ is an ideal of  $\mathfrak{A}_Z$, 
since $q$ is an algebra homomorphism.
To verify that $q(\mathcal I)$ is closed in $\mathfrak{A}_Z$, suppose that $S \in \mathcal K(Z)$ 
with $q(S) \in \overline{q(\mathcal I)}$. Hence there is a bounded sequence 
$(S_n) \subset \mathcal I$ such that 
\[
\Vert q(S) - q(S_n) \Vert = dist(S-S_n, \mathcal A(Z)) \to 0, \quad n \to \infty.
\]
Pick $V_n \in \mathcal A(Z)$ for $n \in \mathbb N$ so that 
$\Vert S - S_n -V_n\Vert \to 0$ as $n \to \infty$.
Here $S_n + V_n \in \mathcal I$ for each $n$, so that $S \in \overline{\mathcal I} = \mathcal I$.
Clearly $q(\mathcal I)  \varsubsetneq \mathfrak{A}_Z$, since otherwise 
$\mathcal I = q^{-1}(q(\mathcal I)) =   \mathcal K(Z)$.
Moreover, a simple verification shows that if  $\mathcal  I_1 \neq \mathcal I_2$, then 
$q(\mathcal  I_1) \neq q(\mathcal  I_2)$.

\smallskip

(ii)  If $U \in q^{-1}(\mathcal J)$ and $S \in \mathcal K(Z)$, then
$q(SU) = q(S)q(U) \in \mathcal J$, so that  $SU \in q^{-1}(\mathcal J)$. 
Similarly $US \in q^{-1}(\mathcal J)$. Moreover,
$\{0\} \neq \mathcal J  \varsubsetneq \mathfrak{A}_Z$ implies that 
$\mathcal A(Z) \varsubsetneq q^{-1}(\mathcal J) \varsubsetneq \mathcal K(Z)$.
\end{proof}

If the closed ideal $\mathcal J$ of $\mathcal K(Z)$ can be written as
$\mathcal J = \overline{\mathcal I(Z)}$ for some Banach operator ideal $\mathcal I$, then
$\mathcal J$ is an ideal of $\mathcal L(Z)$. However, we will also encounter 
closed ideals $\mathcal J$ of $\mathcal K(Z)$ that fail to be an ideal of $\mathcal L(Z)$, 
see Remark \ref{Lideals}.

\medskip

Our first class of examples contains the result that for $p <  2 < q$ there are closed subspaces
$X \subset \ell^p$ and $Y \subset \ell^q$, such that 
$\mathfrak{A}_{X \oplus Y}$ contains at least two non-trivial  incomparable closed ideals. 
The construction is based on the following theorem,
which combines classical factorisation results of Kwapien and Maurey 
(see \cite[Theorem 3.4 and Corollary 3.6]{P2})  with an observation of John \cite[Lemma 2]{Jo}. 
It will be crucial that the equality \eqref{KA}
below is independent of any approximation properties on $X$ or $Y$.
For other results of this type, see \cite[Theorem 2.2]{Go} and the subsequent comments 
for the case of separable reflexive spaces,  as well as the remark on \cite[p. 248]{DJT}.
We refer e.g. to \cite[section 6.2]{AK} for the notions of type and cotype for Banach spaces.
We will require the facts, see \cite[Theorem 6.2.14]{AK}
that $L^p(\mu)$ (as well as all its closed subspaces) has type 
$2$ for $2 \le p < \infty$, 
and cotype $2$ for $1 \le p \le 2$.

\begin{thm}\label{KMJ}
(Kwapien-Maurey-John) Suppose that $X$ has type $2$ and $Y$ has cotype $2$. Then 
\begin{equation}\label{KA}
\mathcal K(X,Y) = \mathcal A(X,Y).
\end{equation}
\end{thm}

\begin{proof}
Since the argument combines two results of different nature we review the relevant ideas for 
completeness.

Suppose that $T \in \mathcal K(X,Y)$ is arbitrary. 
It follows from the classical factorisation results of 
Kwapien and Maurey, see \cite[Corollary 3.6]{P2}, 
that there is a Hilbert space $H$ and bounded operators $A \in \mathcal L(X,H)$,
$B \in \mathcal L(H,Y)$ such that $T = BA$. 

Suppose first that $H$ is separable and let $(P_n)$ 
be the sequence of orthogonal projections
from $H$ onto the linear span $[e_k: 1 \le k \le n]$, where 
$(e_n)$  is a fixed orthonormal basis of $H$. Observe that
\[
\langle  (BP_nA)^{**}x^{**},y^*\rangle = \langle A^{**}x^{**},P^*_nB^*y^* \rangle \to \langle A^{**}x^{**},B^*y^* \rangle = \langle T^{**}x^{**},y^*\rangle
\]
as $n \to \infty$ for any $x^{**} \in X^{**}$ and $y^* \in Y^*$, because
$\Vert P^*_nB^*y^* - B^*y^*\Vert \to 0$ as $n \to \infty$. Since $T$ is a compact operator
this means that  $BP_nA \to T$ weakly in 
$\mathcal K(X,Y)$ as $n \to \infty$, see e.g. \cite[Corollary 3]{K}.  
By Mazur's theorem, for any given $\varepsilon > 0$  there is a finite convex combination 
$\sum_{n = p}^q \lambda_n BP_nA \in \mathcal F(X,Y)$ such that 
\[
\Vert T - \sum_{n = p}^q \lambda_n BP_nA\Vert < \varepsilon.
\]
Hence $T$ is an approximable operator. 

The general case reduces to the separable case by the following elementary observation. 
(Alternatively, one may apply Remark 3 from \cite[page 512]{Jo}, 
but that result depends on more sophisticated facts.)
Namely, if the compact operator $T = BA \in \mathcal K(X,Y)$ factors through a Hilbert
space $H$, then we may actually factor $T$ as  $T = B_0A_0$ through 
a closed separable subspace $H_1$ of $H$. 
Indeed, it suffices to show that $T$ is approximable considered as an operator 
$X \to \overline{T(X)}$, where 
$\overline{TX} \subset Y$ is a separable subspace, since $\overline{TB_X}$ is separable
by the compactness of $T$. Hence we may also suppose that $Y$ is separable.
We first factor $B = \widehat{B}P$ through $H_1 =: Ker(B)^{\perp}$, where $P$
is the orthogonal projection of $H$ onto $H_1$ and $\widehat{B} = B_{| Ker(B)^{\perp}}$. 
Thus $T = B_0A_0$, where $B_0 = \widehat{B}$ is an injective operator 
$H_1 \to Y$ and $A_0 = PA$.
Finally,  it follows that  $H_1$ is separable from
the general fact stated separately in Lemma \ref{separable} below.
\end{proof}

Above we applied the following general observation, which we include for 
completeness.

\begin{lma}\label{separable}
Let $Z$ and $W$ be Banach spaces such that $Z$ is reflexive and $W$ is separable.
If there is a bounded linear injection $S: Z \to W$, 
then $Z$ is a separable space.
\end{lma}

\begin{proof}
Fix an isometric embedding
$J: W \to \ell^\infty$ and let $D: \ell^\infty \to \ell^2$ be the injective diagonal operator
$(x_n) \mapsto (a_nx_n)$, where $(a_n) \in \ell^2$ and $a_n \neq 0$ for all $n$. 
Then $U = DJS$ is a bounded linear injection $Z \to \ell^2$. 
Since $Z$ is reflexive, the range $U^*(\ell^2)$
is norm-dense in $Z^*$ by the Hahn-Banach and the Mazur theorems, 
so that $Z^*$ is separable.
\end{proof}

\smallskip

As a brief digression we note that the result of John \cite{Jo} used in Theorem \ref{KMJ} 
also yields that there are no
$\mathcal A^c$-universal operators for the class $\mathcal A^c$ of the 
non-approximable operators, which was mentioned as a problem in 
\cite[2.1]{BC}. Let $\mathcal I$ be a Banach operator ideal. 
Recall from \cite[1.12]{DJP} that the operator 
$U \in \mathcal L(X,Y)$, where $U \notin \mathcal I(X,Y)$,
is $\mathcal I^c$-universal
if for every Banach space $Z$, $W$ and 
for every bounded operator 
$V \in  \mathcal L(Z,W) \setminus  \mathcal I(Z,W)$ 
there are bounded operators $A \in \mathcal L(X,Z)$
and $B \in \mathcal L(W,Y)$ such that $U = BVA$. For instance, Johnson \cite{Johnson71}
showed that the canonical inclusion $J: \ell^1 \to \ell^\infty$ is a $\mathcal K^c$-universal
operator. We refer to the recent paper by Beanland
and Causey \cite{BC} for a systematic study of universal factoring operators 
for various classes of operators.

\begin{cor}\label{universal}
There are no  $\mathcal A^c$-universal operators.
\end{cor}

\begin{proof}
Suppose to the contrary that $U \in \mathcal L(X,Y)$ is a $\mathcal A^c$-universal  
operator. Fix a compact non-approximable operator  
$V \in \mathcal K(Z,W) \setminus \mathcal A(Z,W)$ 
for suitable Banach spaces $Z$ and $W$. 
By assumption there are bounded operators $A$ and $B$ for which 
$U = BVA$, so that $U \in \mathcal K(X,Y)$. Moreover, by assumption  
$U$ also factors through 
the identity $I_{\ell^2}$, so that $U$ is a compact operator that factors through
a Hilbert space.  It follows from \cite[Lemma 2]{Jo} 
that $U \in \mathcal A(X,Y)$, which is not possible.
\end{proof}

\smallskip

We will in the sequel several times require various classical examples 
related to the existence of closed subspaces of $\ell^p$-spaces that fail the A.P.  
To avoid repetition we 
collect these results here, together with their sources, for convenient reference.

\begin{facts}\label{ap1}
Let $1 \le p < \infty$ and $p \neq 2$. Then

(i) there are closed subspaces $X \subset \ell^p$ and $X \subset c_0$ 
that fail the A.P.,  and

(ii) there are closed subspaces $Z \subset \ell^p$ and $Z \subset c_0$ 
such that $\mathcal A(Z) \varsubsetneq  \mathcal K(Z)$.
\end{facts}

For part (i) recall that the first examples of closed subspaces $X$ failing the A.P. 
were constructed by Enflo \cite{E} for $c_0$, 
by Davie \cite{Da73} for $\ell^p$ and  $2 < p < \infty$, and
by Szankowski \cite{Sz78} for $\ell^p$ and $1 \le p < 2$.
We also refer to \cite[Section 2.d]{LT1}, \cite[Section 1.g]{LT2} and \cite[Section 10.4]{Pie}
for systematic expositions.

Concerning (ii) Alexander \cite{A} found closed subspaces $Z \subset \ell^p$ 
for $2 < p < \infty$ such that 
$\mathcal A(Z) \varsubsetneq  \mathcal K(Z)$, 
and she observed that similar examples can be deduced from 
the compact factorisation result \cite[Theorem 7.4]{F} of Figiel. 
Moreover, a general result of Bachelis \cite[Theorem 2']{B} immediately implies (ii), 
once the examples in (i) are known. 

We point out that in part (ii) the case $Z \subset c_0$ can also be deduced from earlier 
factorisation results. In fact, if the closed subspace $X \subset c_0$ fails to have the A.P.,
then by \cite[Theorem 1.e.4]{LT1} there is a Banach space $Y$ and 
an operator $T \in \mathcal K(Y,X) \setminus \mathcal A(Y,X)$.
By a compact factorisation theorem of Terzio\u{g}lu \cite{T} (see also Randtke \cite{R})
we may factor $T$ compactly through a closed 
subspace of $c_0$, that is, there is a closed subspace $Z_0 \subset c_0$ and 
$A \in \mathcal K(Y,Z_0)$, $B \in \mathcal K(Z_0,X)$ such that $T = BA$.
Consider $Z = Z_0 \oplus X \subset c_0$ and define $U \in \mathcal L(Z)$ by 
\[
U(x,y) = (0,Bx), \quad (x,y) \in Z.
\]
 It follows that 
$U \in \mathcal K(Z) \setminus \mathcal A(Z)$,
since $B$ cannot be an approximable operator.

\medskip

We will often use the operator matrix notation
\[
U = \left( \begin{array}{ccc}
U_{11} & U_{12}\\
U_{21} & U_{22}\\
\end{array} \right)
\]
for bounded operators $U \in \mathcal L(X\oplus Y)$ on direct sums $X \oplus Y$.
The alternative notation $(U)_{i,j} = U_{ij}$ for  $U \in \mathcal L(X\oplus Y)$ and
$i, j = 1,2$ will also be used where more appropriate.
The operator ideal property of $\mathcal K$ implies that
$U \in \mathcal K(X\oplus Y)$ if and only if each component operator
$U_{ij} \in \mathcal K$, and a similar fact holds for the components
of $U \in \mathcal I(X\oplus Y)$ for any Banach operator ideal  $\mathcal I$.
Let $\mathcal I_{11} \subset \mathcal K(X)$, $\mathcal I_{12} \subset \mathcal K(Y,X)$,
$\mathcal I_{21} \subset \mathcal K(X,Y)$ and $\mathcal I_{22}  \subset \mathcal K(Y)$
be given classes of operators.
We introduce the convenient notation 
\[
\left( \begin{array}{ccc}
\mathcal I_{11} & \mathcal I_{12} \\
\mathcal I_{21}  & \mathcal I_{22} \\
\end{array} \right) =: \big \{ U = \left( \begin{array}{ccc}
U_{11} & U_{12}\\
U_{21} & U_{22}\\
\end{array} \right) \in \mathcal K(X \oplus Y):   U_{ij} \in \mathcal I_{ij} \textrm{ for } 
i, j = 1, 2 \big \},
\]
for the resulting class of compact operators on $X \oplus Y$. 

We proceed to construct a class of direct sums $Z = X \oplus Y$, for which
$\mathfrak{A}_Z$ admits non-trivial closed ideals.
The fact  that the quotient  algebra $\mathfrak{A}_{X\oplus Y}$
has a lower triangular form in
the operator matrix representation will play a crucial role. There is an 
analogy with the classical result that  the Banach algebra
$\mathcal L(\ell^p \oplus \ell^q)$ contains two 
incomparable maximal closed ideals for $p < q$, see e.g. \cite[5.3.2]{Pie}.
 
\begin{thm}\label{ideals} 
Suppose that $X$ and $Y$ are Banach spaces such that 
$X$ has cotype $2$, $Y$ has type $2$, as well as  $\mathcal A(X) \varsubsetneq \mathcal K(X)$
and  $\mathcal A(Y) \varsubsetneq \mathcal K(Y)$. 
Let
\[
\mathcal I =: \left( \begin{array}{ccc}
\mathcal K(X) & \mathcal A(Y,X) \\
\mathcal K(X,Y)  &  \mathcal A(Y) \\
\end{array} \right)  \textrm{ and } 
\mathcal J =:  \left( \begin{array}{ccc}
\mathcal A(X) & \mathcal A(Y,X) \\
\mathcal K(X,Y)  &  \mathcal K(Y) \\
\end{array} \right),
\]
where  $\mathcal K(Y,X) = \mathcal A(Y,X)$  in view of  Theorem \ref{KMJ}.

Then $\mathcal I$ and $\mathcal J$ are non-trivial incomparable 
closed ideals of $\mathcal K(X\oplus Y)$, 
and $q(\mathcal I)$ and $q(\mathcal J)$ are non-trivial incomparable 
closed ideals of  $\mathfrak{A}_{X\oplus Y}$. 
In particular, for  $1 \le p <  2 < q < \infty$ there are closed subspaces 
$X \subset \ell^p$ and $Y \subset \ell^q$ for which $\mathfrak{A}_{X\oplus Y}$
contains (at least) two non-trivial incomparable closed ideals. 
\end{thm}

\begin{proof}
We show that $\mathcal I$ and $\mathcal J$  
are actually closed ideals of $\mathcal L(X\oplus Y)$.
Clearly $\mathcal I$ and $\mathcal J$ are closed subspaces of 
$\mathcal K(X\oplus Y)$.

Let $U =  \left( \begin{array}{ccc}
U_{11} & U_{12}\\
U_{21} & U_{22}\\
\end{array} \right), V =  \left( \begin{array}{ccc}
V_{11} & V_{12}\\
V_{21} & V_{22}\\
\end{array} \right) \in \mathcal L(X\oplus Y)$, so that
\begin{equation}\label{prod}
UV = 
 \left( \begin{array}{ccc}
U_{11}V_ {11} + U_{12}V_{21}& U_{11}V_{12} + U_{12}V_{22}\\
U_{21}V_ {11} + U_{22}V_{21} &  U_{21}V_ {12} + U_{22}V_{22}\\
\end{array} \right).
\end{equation}

We claim that $UV \in \mathcal I$ and $VU \in \mathcal I$
for arbitrary $U \in  \mathcal I$ and $V \in \mathcal L(X\oplus Y)$.
Note first that both $UV$ and $VU$ belong to  $\mathcal K(X\oplus Y)$, 
so that by Theorem \ref{KMJ} the component 
\[
(UV)_{1,2} \in \mathcal K(Y,X) = \mathcal A(Y,X),
\]
and similarly $(VU)_{1,2} \in  \mathcal A(Y,X)$.
It remains to verify that the
components $(UV)_{2,2}  \in \mathcal A(Y)$ and $(VU)_{2,2} \in \mathcal A(Y)$.  
By \eqref{prod} we know that
\[
(UV)_{2,2} = U_{21}V_{12} + U_{22}V_{22}.
\]
In view of the factorisation theorem of Kwapien and Maurey 
(see \cite[Theorem 3.4 and Corollary 3.6]{P2}) we may factor $V_{12} = BA$, where 
$B \in \mathcal L(H,X)$ and $H$ is a Hilbert space. Since $U_{21} \in \mathcal K(X,Y)$
it follows that $U_{21}B \in \mathcal K(H,Y) = \mathcal A(H,Y)$
and $U_{21}V_{12} = U_{12}BA \in \mathcal A(Y)$.
Consequently $(UV)_{2,2} \in \mathcal A(Y)$, since $U_{22} \in \mathcal A(Y)$ by assumption. Moreover,  the corresponding component
$(VU)_{2,2} = V_{21}U_{12} + V_{22}U_{22}  \in \mathcal A(Y)$, since 
$U_{12}$ and $U_{22}$ are approximable operators. 

In the case of $\mathcal J$ it remains to show that the (1,1)-components
of $UV$ and $VU$ are approximable for $U \in  \mathcal J$ and $V \in \mathcal L(X\oplus Y)$. 
From \eqref{prod} we have
\[
(VU)_{1,1} = V_{11}U_{11} + V_{12}U_{21},
\]
where $U_{11} \in \mathcal A(X)$ by assumption.
As above one uses the 
Kwapien-Maurey factorisation theorem to deduce that 
$V_{12}U_{21} \in \mathcal A(X)$. In addition, 
$(UV)_{1,1} =  U_{11}V_{11} + U_{12}V_{21} \in \mathcal A(X)$
since the operators $U_{11}$ and $U_{12}$ are approximable.

From our assumptions there are operators
$U_{11} \in  \mathcal K(X) \setminus  \mathcal A(X)$ and
$V_{22} \in  \mathcal K(Y) \setminus  \mathcal A(Y)$, 
so that 
\[
U = \left( \begin{array}{ccc}
U_{11} & 0\\
0 & 0\\
\end{array} \right) \in \mathcal I \setminus \mathcal J 
\textrm{ and } V = \left( \begin{array}{ccc}
0 & 0\\
0 & V_{22}\\
\end{array} \right) \in \mathcal J \setminus \mathcal I.
\] 
Hence $\mathcal  I$ are $\mathcal  J$ are incomparable ideals, and  it follows that
\[
\mathcal A(X \oplus Y)  \varsubsetneq \mathcal  I  \varsubsetneq \mathcal K(X \oplus Y),
\quad 
\mathcal A(X \oplus Y)  \varsubsetneq \mathcal  J  \varsubsetneq \mathcal K(X \oplus Y).
\]
Proposition \ref{quotient} yields that 
$q(\mathcal I)$ and $q(\mathcal J)$ are non-trivial incomparable 
closed ideals of  $\mathfrak{A}_{X\oplus Y}$.

For the final claim recall from Facts \ref{ap1}.(ii) 
that for $1 \le p <  2 < q < \infty$ there are  closed subspaces 
$X \subset \ell^p$ and $Y \subset \ell^q$ 
such that $\mathfrak{A}_{X} \neq \{0\}$ and $\mathfrak{A}_{Y} \neq \{0\}$.
Moreover, recall that  $X \subset \ell^p$ has cotype $2$ for 
$1 \le p < 2$ and  $Y \subset \ell^q$ has type $2$ for $2 < q < \infty$. 
\end{proof}

It is relevant to ask how many non-trivial closed ideals of 
$\mathcal K(X\oplus Y)$ we may construct for direct sums $X \oplus Y$ 
belonging to the class of spaces in Theorem \ref{ideals}.
For instance, the closed ideal $\mathcal I \cap \mathcal J$ of 
$\mathcal K(X\oplus Y)$ is non-trivial if and only if  
$\mathcal A(X,Y) \varsubsetneq \mathcal K(X,Y)$.
In Example \ref{8ideal} we will construct a
direct sum $X \oplus Y$ from this class of spaces for which $\mathcal K(X\oplus Y)$, 
and consequently also $\mathfrak{A}_{X \oplus Y}$,
contains at least 8 non-trivial closed ideals. Such examples 
require more preparation, including a study of the strict inclusion  
$\mathcal A(X,Y) \varsubsetneq \mathcal K(X,Y)$
among closed subspaces $X \subset \ell^p$ and $Y \subset \ell^q$ for $p \neq q$, 
as well as non-classical approximation properties associated to Banach operator ideals. 

\medskip

Recall that the Banach algebra $A$ is \textit{nilpotent} if there is $m \in \mathbb N$ such that 
$x_1 \cdots x_m = 0$ for all $x_1,\ldots, x_m \in A$.
It follows from a general result, see   \cite[Proposition 1.5.6.(iv)]{D00}, that a 
non-nilpotent complex radical Banach algebra is infinite-dimensional. 
In the sequel we also wish to apply this fact to 
the real quotient algebra $\mathfrak{A}_X = \mathcal K(X)/\mathcal A(X)$
in the case of real Banach spaces $X$.
One reason is that the problem 
of the closed ideals of $\mathfrak{A}_X$, that is, 
the closed ideals $\mathcal A(X) \subset \mathcal J \subset \mathcal  K(X)$,  
is also relevant for real spaces $X$.
To justify this application  we must verify that
for real Banach spaces $X$ the real Banach algebra $\mathfrak{A}_X$ is radical 
in the interpretation \eqref{rrad} from Section \ref{intro},  which is one of the equivalent 
conditions of radicality for complex Banach algebras. 
Actually, we will need a general version 
for quotient algebras which are formed from nested closed ideals contained 
in the class of inessential operators. 

Recall that the operator $S \in \mathcal L(X,Y)$ is \textit{inessential}, denoted
$S \in \mathcal{R}(X,Y)$ if $I_X - TS$ is a Fredholm operator for all
$T \in \mathcal L(Y,X)$, that is,
the kernel $Ker(I_X - TS)$ is finite-dimensional and the range $Im(I_X - TS)$
has finite codimension in $X$. The closed Banach operator ideal $\mathcal{R}$
was introduced by Kleinecke \cite{Kl}. Thus  
$\mathcal K(X) \subset \mathcal R(X)$, and it is known that
$\mathcal{R}(X)$ is the largest closed ideal of $\mathcal L(X)$ 
for which Atkinson's characterisation of Fredholm operators 
holds for any Banach space $X$  (see below). 
Recall that the Banach operator ideal  $\mathcal I$  is closed if the components 
$\mathcal{I}(X,Y)$ are closed in the operator norm 
for all spaces $X$ and $Y$.

\begin{prop}\label{real}
Let $X$ be an infinite-dimensional real Banach space.  
If $\mathcal I$ and $\mathcal J$ are closed Banach operator ideals such that
$\mathcal A(X) \subset \mathcal I(X) \subset \mathcal J(X) \subset \mathcal{R}(X)$,
then the real quotient algebra 
\[
\mathfrak{J}^{\mathcal I}_X  = : \mathcal J(X)/\mathcal I(X) 
\]
is radical in the sense of \eqref{rrad}. In particular, 
 $\mathfrak{A}_X$ is a real radical Banach algebra.
\end{prop}

\begin{proof}
Put $\mathbb K = \mathbb R$. Recall that 
the abstract unitisation $(\mathfrak{J}^{\mathcal I}_X)^{\#}$ of $\mathfrak{J}^{\mathcal I}_X$ 
consists of $\mathbb K \oplus \mathfrak{J}^{\mathcal I}_X$ equipped with the product
\[
(\alpha,S+\mathcal I(X)) \cdot (\beta,T+\mathcal I(X)) =
(\alpha \beta, \alpha T + \beta S + ST + \mathcal I(X)), \quad 
\alpha,  \beta \in \mathbb K,\  S, T \in \mathcal J(X),
\]
and the algebra norm 
$\Vert (\alpha,S+\mathcal I(X))\Vert = \vert \alpha \vert + \Vert S + \mathcal I(X)\Vert$.
Observe next that the unitisation $(\mathfrak{J}^{\mathcal I}_X)^{\#}$  can also be concretely identified with the closed subalgebra
\[
B =: \big( \mathbb KI_X  \oplus \mathcal J(X)\big)/\mathcal I(X) 
\]
of $\mathcal L(X)/\mathcal I(X)$. (Since $\mathcal J(X) \subset \mathcal R(X)$ 
this observation will connect the unitisation to 
classical Fredholm theory, which is independent of the scalar field.)
To see this identification note that the map 
\[
(\alpha,S+\mathcal I(X)) \mapsto \alpha I_X + S + \mathcal I(X)
\]
defines an algebra isomorphism $\theta: \mathbb K \oplus \mathfrak{J}^{\mathcal I}_X \to B$, since
\[
(1/3)(\vert \alpha \vert + \Vert S + \mathcal I(X)\Vert) \le \Vert \alpha I_X + S + \mathcal I(X) \Vert
\le \vert \alpha \vert + \Vert S + \mathcal I(X)\Vert
\]
for $\alpha \in \mathbb K$ and $S \in \mathcal J(X)$. In fact,  
for $S \in \mathcal J(X)$ we have
\[
\vert \alpha \vert =  \Vert \alpha I_X + \mathcal J(X) \Vert  =
\Vert \alpha I_X + S + \mathcal J(X) \Vert \le 
\Vert  \alpha I_X + S + \mathcal I(X) \Vert,
\]
so that $\vert \alpha \vert + \Vert S + \mathcal I(X)\Vert \le 3 
\Vert \alpha I_X + S + \mathcal I(X) \Vert$ by the triangle inequality.
Above we used that $\Vert I_X + \mathcal J(X) \Vert = 1$. Namely, if $\Vert I_X - U \Vert < 1$
for some $U \in \mathcal J(X)$, then 
$U = I_X - (I_X - U)$ is invertible.
This is impossible since  $X$ is infinite-dimensional
and  $U \in  \mathcal J(X) \subset \mathcal R(X)$.

Next let $\alpha \neq 0$ and $S \in \mathcal J(X)$ be arbitrary. To identify 
the inverse $(\alpha I_X + S + \mathcal I(X))^{-1}$ in $B$, 
let $\pi: \mathcal L(X) \to \mathcal L(X)/\mathcal J(X)$ be the quotient map.
Recall that  an operator $U \in \mathcal L(X)$ is invertible modulo 
$\mathcal A(X)$ (that is, a Fredholm operator) if and only if it is invertible modulo 
$\mathcal R(X)$, see e.g. \cite[Theorem 2]{Kl} or \cite[section 26.3 and 26.7.2]{Pie}. 
Note that this result holds equally well for real Banach spaces.
We know that  
$\pi(\alpha I_X + S) = \alpha I_X + \mathcal J(X)$
has the inverse $\alpha^{-1} I_X + \mathcal J(X)$
in $\mathcal L(X)/\mathcal J(X)$.
Hence the above invertibility fact implies that there is  
$V \in \mathcal J(X)$ and $R_1, R_2 \in \mathcal I(X)$ for which
\[
(\alpha^{-1}I_X + V)(\alpha I_X + S) = I_X + R_1 \textrm{ and } 
(\alpha I_X + S) (\alpha^{-1}I_X + V) = I_X + R_2.
\]
In other words, the inverse 
$(\alpha I_X + S + \mathcal I(X))^{-1} = \alpha^{-1}I_X + V + \mathcal I(X)$
belongs to $B$, 
so that  $\sigma_{\mathbb R}(S + \mathcal I(X)) = \{0\}$ 
for all $S \in \mathcal J(X)$ (that is, condition \eqref{rrad} holds).

Finally, $\mathfrak{A}_X$ is obtained for $\mathcal I(X) =
\mathcal A(X)$ and $\mathcal J(X)= \mathcal K(X)$.
\end{proof}

The above argument is also valid for the complex scalars $\mathbb K = \mathbb C$, 
so that $\mathfrak{J}^{\mathcal I}_X  = \mathcal J(X)/\mathcal I(X)$
is a radical Banach algebra in the classical sense.
We also point out that the non-zero
real quotient algebras $\mathfrak{A}_X \neq \{0\}$ (as well as 
$\mathfrak{J}^{\mathcal I}_X \neq \{0\}$)
cannot have a unit element in view of Proposition \ref{real}. 
In fact, if $S + \mathcal A(X)$ were the unit in $\mathfrak{A}_X$ for some 
$S \in \mathcal K(X)$, then $I_X - S + \mathcal A(X)$ is invertible in the
unitisation $B$, because $\sigma_{\mathbb R}(-S + \mathcal A(X)) = \{0\}$. 
Let $\alpha I_X + T + \mathcal A(X)$ be the inverse in $B$, so that
\begin{align*}
I_X + \mathcal A(X) & = (I_X - S + \mathcal A(X))  (\alpha I_X + T + \mathcal A(X)) 
= \alpha I_X - \alpha S +  T - ST + \mathcal A(X) \\
& =  \alpha I_X - \alpha S  + \mathcal A(X),
\end{align*}
since $ST - T \in  \mathcal A(X)$ by assumption.
This implies that $\alpha = 1$ and  $S \in \mathcal A(X)$, which contradicts the assumption.

\medskip

We showed in  \cite[Proposition 3.1.(i)]{TW} that if the Banach space $X$
has the B.C.A.P., but fails the A.P.,
then there is a compact operator $U \in \mathcal K(X)$ such that 
$U^m \notin \mathcal A(X)$ for any $m \in \mathbb N$. In particular, the quotient
algebra  $\mathfrak{A}_X$ is non-nilpotent.
In  \cite[Proposition  2.2 and Corollary 2.4]{TW} we constructed linear isomorphic 
embeddings $\psi: c_0 \to \mathfrak{A}_X$ for certain closed subspaces $X \subset \ell^p$, 
where $1 \le p < \infty$ and $p \neq 2$, and $X \subset c_0$. 
However, by  \cite[Proposition 2.5]{TW}
the embedding $\psi$ cannot preserve any of the multiplicative structure of $c_0$, and
the construction does not ensure that $\mathfrak{A}_X$ is non-nilpotent. 
We next improve and complement these results from \cite{TW},  
and show there are also closed subspaces  $X \subset \ell^p$, where 
$1 \le p < \infty$ and $p \neq 2$, and $X \subset c_0$, 
such that the quotient algebra $\mathfrak{A}_X$ is non-nilpotent and infinite-dimensional.
These subspaces will also be used in later examples.
By contrast with \cite[section 2]{TW} the subsequent construction in Theorem \ref{nil} 
will be  based on  
a general compact factorisation result of Bachelis \cite{B}, which allows for 
greater generality and additional features.

\smallskip

Let $\mathcal E$ be a real Banach space with a $1$-unconditional basis $(e_n)$, and
suppose that $(E_n)$ is a sequence of Banach spaces. The corresponding 
$\mathcal E$-direct sum  is defined by
\[
\big( \oplus_{n \in \mathbb N} E_n \big)_{\mathcal E} 
= \{(x_n): x_n \in E_n, \ \sum_{n=1}^\infty \Vert x_n\Vert e_n \textrm{ converges in } \mathcal E\}.
\]
It is not difficult to check by using \cite[Proposition 1.c.7]{LT1} that 
\[
\Vert (x_n)\Vert_{\mathcal E} = 
\big \Vert \sum_{n=1}^\infty \Vert x_n\Vert e_n \big \Vert_{\mathcal E}, \quad
(x_n) \in \big( \oplus_{n \in \mathbb N} E_n \big)_{\mathcal E},
\]
defines a complete norm in $\big( \oplus _{n \in \mathbb N} E\big)_{\mathcal E}$.
We use $Y \approx Z$ for linearly isomorphic 
Banach spaces $Y$ and $Z$.

\begin{thm}\label{nil}
Suppose that the Banach space $E$ has the B.A.P., and that 
\[
\big( \oplus _{n \in \mathbb N} E\big)_{\mathcal E} \approx E
\]
for some real Banach space $\mathcal E$  which has a $1$-unconditional basis. 
If  $E$ has a closed subspace $X_0$ that fails the A.P.,
then there is a closed subspace $X \subset E$ such that 
$\mathfrak{A}_X$ is infinite-dimensional,  and for which there is an operator $U \in \mathcal K(X)$
that satisfies $U^m \notin \mathcal{A}(X)$ for all $m \in \mathbb N$.
In particular,  $\mathfrak{A}_X$ is a non-nilpotent quotient algebra.

The assumptions apply e.g. to  $E = \ell^p$ for $1 \le p < \infty$ and
$p \neq 2$, or $E = c_0$. 
\end{thm}

\begin{proof}
Since $X_0$ fails the A.P. there is  a Banach space $W$ and a 
compact non-approximable operator $T_0: W \to X_0$, see e.g.  \cite[Theorem 1.e.4]{LT1}. 
By Bachelis' factorisation theorem
 \cite[Theorem 2']{B} there is a closed subspace 
$Z_1$ of $E$ and a compact factorisation 
$T_0 = B_1A_1$, that is,  $A_1 \in \mathcal K(W, Z_1)$ and $B_1 \in \mathcal K(Z_1,X_0)$.
Note that $B_1 \notin \mathcal A(Z_1,X_0)$,
since $T_0$ is not an approximable operator. 

By successive applications of \cite[Theorem 2']{B} we obtain a
sequence $(Z_n)$ of closed subspaces of $E$ as well as sequences $(A_n)$ and $(B_n)$
of compact operators, such that
\[
B_n = B_{n+1}A_{n+1}, \ \textrm{ where } A_{n+1} \in \mathcal K(Z_n,Z_{n+1})\  \textrm{ and }  
B_{n+1} \in \mathcal K(Z_{n+1},X_0) \setminus \mathcal A(Z_{n+1},X_0) 
\]
for all $n \in \mathbb N$. 

Let $X = \big( X_0 \oplus ( \oplus_{n \in \mathbb N} Z_n)\big)_{\mathcal E}$.
By assumption $X$ is, up to a linear isomorphism, a closed linear subspace of $E$.
Let $J_n$ denote the natural inclusion maps and $P_n$ the natural 
coordinate projections for $n \ge 0$,
such that $P_0: X \to X_0$ and $J_0: X_0 \to X$, whereas 
$P_n: X \to Z_n$ and $J_n: Z_n \to X$ for $n \ge 1$. 
(Here we canonically identify $X_0$ and $Z_n$
with closed subspaces of $X$.)
Define operators $\widehat{B}_n \in \mathcal L(X)$ 
for $n \in \mathbb N$
and $\widehat{A}_n \in \mathcal L(X)$ for $n \ge 2$ as follows:
\[
\widehat{B}_n(x,z_1,z_2,\ldots ) = (B_nz_n,0,0,\ldots ) \textrm{ and  }
\widehat{A}_n(x,z_1,\ldots) = (0,0,\ldots,A_{n}z_{n-1},0,\ldots ),
\]
where $A_{n}z_{n-1}\in Z_{n}$  (in the $n$:th position in 
$\oplus_{k \in \mathbb N} Z_k$).
This means that $\widehat{B}_n = J_0B_nP_n \in \mathcal K(X)$ for all
$n \in \mathbb N$ and
$\widehat{A}_n = J_nA_nP_{n-1}  \in \mathcal K(X)$ for all $n \ge 2$. 
Moreover,
\begin{equation}\label{iterate}
\widehat{B}_1 = \widehat{B}_2 \widehat{A}_2 = \widehat{B}_3 \widehat{A}_3 \widehat{A}_2 =
\ldots =  \widehat{B}_n \circ \big( \prod_{k=1}^{n-1} \widehat{A}_{n+1-k} \big)
\end{equation}
for $n \in \mathbb N$. Namely, by successive evaluations one gets that
\begin{align*}
\widehat{B}_n \widehat{A}_n & \widehat{A}_{n-1} \cdots \widehat{A}_2(x_1,z_1,z_2,\ldots)
= (B_nA_n \ldots A_2z_1,0,\ldots )\\ 
& = (B_1z_1,0,\ldots ) = \widehat{B}_1(x_1,z_1,z_2,\ldots)
\end{align*}
for all $(x_1,z_1,\ldots) \in X$. It follows from \eqref{iterate}  that the 
quotient algebra $\mathfrak{A}_X$ is not nilpotent, since $\widehat{B}_1 \notin \mathcal A(X)$. 

We are now in position to repeat the argument of \cite[Proposition  3.1]{TW}.
In fact, $\mathfrak{A}_X$ is a  non-nilpotent radical Banach algebra, 
which is infinite-dimensional by \cite[Proposition  1.5.6.(iv)]{D00}. 
Note that for real Banach spaces $X$ 
we need Proposition \ref{real} to ensure that 
$\mathfrak{A}_X$ is a real radical Banach algebra  in the sense of  \eqref{rrad}.
Moreover, by the Baire category argument from \cite{G}
there is an operator $U \in \mathcal K(X)$ 
such that $U^m \notin \mathcal A(X)$ for any $m \in \mathbb N$. 
\end{proof}
 
In the case of complex scalars the operator $U \in \mathcal K(X)$
from Theorem \ref{nil} has the stronger property that 
$U^n - U^m \notin \mathcal A(X)$ for all $n \neq m$.
Namely, if $U^{n+k} - U^n  \in \mathcal A(X)$ for some $n, k \in \mathbb N$, then
by iteration $U^{n+ ks} - U^n \in \mathcal A(X)$ for all $s \in \mathbb N$. Conclude that
\[
\lim_{s \to \infty} \Vert U^{n+ ks} + \mathcal A(X)\Vert^{1/(n+ks)} = 
\lim_{s \to \infty}  \Vert U^n + \mathcal A(X)\Vert^{1/(n+ks)} =  1,
\]
since $U^n \notin \mathcal A(X)$.
This contradicts the fact that $\mathfrak{A}_X$ is a radical Banach algebra, because 
$\lim_{m\to\infty} \Vert V^m + \mathcal A(X)\Vert^{1/m} = 0$  for all $V \in \mathcal K(X)$
by the spectral radius formula.

\medskip
 
Theorem \ref{nil} enables us to revisit the setting of 
\cite[Proposition  2.2 and Corollary 2.4]{TW}, and to obtain
linear isomorphic embeddings $c_0 \to \mathfrak{A}_X$ that display quite extreme properties
for particular closed subspaces  $X \subset \ell^p$, 
where $1 \le p < \infty$ and $p \neq 2$, or $X \subset c_0$.

\begin{ex} 
There is a closed subspace  $X \subset \ell^p$
for $1 \le p < \infty$ and $p \neq 2$, or $X \subset c_0$ for which

(i) there is a linear isomorphic embedding 
$\theta: c_0 \to \mathfrak{A}_X$, such that
$\theta(a)^n \neq 0$ for all $a \neq 0$ and $n \in \mathbb N$, or

(ii) there is  a linear isomorphic embedding 
$\psi: c_0 \to \mathfrak{A}_X$, such that
$\psi(a) \psi(b) = 0$ for all $a, b \in c_0$. In particular,
the closed subalgebra 
$\mathcal A(\psi(c_0))$ of $\mathfrak{A}_X$ generated by $\psi(c_0)$ satisfies 
\[
\mathcal A(\psi(c_0)) = \psi(c_0).
\]
\end{ex}

\begin{proof} 
For unity of notation we only construct the desired closed subspaces $X \subset \ell^p$. 
The case where $X \subset c_0$ is similar.

\smallskip

(i) By Theorem \ref{nil} 
there is a closed linear subspace $Y \subset \ell^p$ for $p \neq 2$
as well as $U \in \mathcal K(Y)$ such that $U^m \notin \mathcal A(Y)$ for all $m \in \mathbb N$.
Consider  $X = \big( \oplus_{n\in \mathbb N} Y \big)_{\ell^p}$,
which can be identified with a closed subspace of $\ell^p$. 
Define $\beta: c_0 \to \mathcal K(X)$ by 
\[
\beta(a) = \sum_{k=1}^\infty a_kU_k, \quad a = (a_k) \in c_0,
\]
where $U_k$ denotes the operator $U$ defined on the $k$:th copy of $Y$ in $X$. 
Let $a = (a_k) \in c_0$ be arbitrary. 
Since $U_k \circ U_r = 0$ for all $k \neq r$ by definition, we get that
\[
\big(\sum_{k=1}^m a_kU_k\big) \circ \big(\sum_{r=1}^m a_kU_k\big) = 
\sum_{k=1}^m a_k^2 U_k^2 
\]
for all $m \in \mathbb N$. Since $(a_k) \in c_0$ we may pass  to the limit in $\mathcal K(X)$ 
as $m \to \infty$, and obtain that 
$\beta(a)^2 = \beta(a) \cdot \beta(a) =  \sum_{k=1}^\infty a_k^2 U_k^2$. 
By induction we get that
\[
\beta(a)^n = \sum_{k=1}^\infty a_k^n U_k^n, \quad a = (a_k) \in c_0.
\]
Suppose next that $a = (a_k) \neq 0$ and pick $k \in \mathbb N$ such that $a_k \neq 0$. 
Clearly $\beta(a)^n \notin \mathcal A(X)$, since the $k$:th term 
$a_k^n U_k^n \notin \mathcal A(X)$.

Finally, one verifies as in the proof of  \cite[Proposition 2.2]{TW} that 
$\theta = q \circ \beta$ is a linear isomorphic embedding $c_0 \to \mathfrak{A}_X$,
where $q: \mathcal K(X) \to \mathfrak{A}_X$ is the quotient map. 
From the above computation it follows that
$\theta(a)^n \neq 0$ whenever $a \neq 0$ and $n \in \mathbb N$.

\smallskip

(ii) The argument modifies the original construction in \cite[section 2]{TW}.
According to Facts \ref{ap1}.(ii) we may pick a closed linear subspace $Z \subset \ell^p$
and an operator $U \in \mathcal K(Z) \setminus \mathcal A(Z)$.
Consider $Y = Z \oplus Z$ and define $V \in \mathcal K(Y)$ by 
$V(x,y) = (0,Ux)$ for $(x,y) \in Z \oplus Z$. Then $V \notin \mathcal A(Y)$ and $V^2 = 0$.

Consider $X = \big( \oplus_{n\in \mathbb N} Y \big)_{\ell^p}$, and
define $\beta: c_0 \to \mathcal K(X)$ by 
\[
\beta(a) = \sum_{k=1}^\infty a_kV_k, \quad a = (a_k) \in c_0,
\]
where $V_k$ denotes the operator $V$ on the $k$:th copy of $Y = Z \oplus Z$ in $X$. 
As above $\psi = q \circ \beta$ is a linear isomorphic embedding
$c_0 \to  \mathfrak{A}_X$. Let $a = (a_k), b = (b_k) \in c_0$
be arbitrary. Since $V_k \circ V_r = 0$ for all $k, r \in \mathbb N$ by definition, we get that
\[
(\sum_{k=1}^m a_kV_k) \circ (\sum_{r=1}^m b_kV_k) = \sum_{k=1}^m a_kb_k V_k^2 
= 0
\]
for all $m \in \mathbb N$. Deduce that in the limit
$\beta(a) \beta(b) = 0$, so  that  $\psi(a) \psi(b) = 0$. 

Finally, finite linear combinations of the products 
$\psi(u_1) \cdots \psi(u_s)$,
where $u_1, \ldots, u_s \in c_0$ and $s \in \mathbb N$, are dense in 
$\mathcal A(\psi(c_0))$. 
This implies that $\mathcal A(\psi(c_0)) = \psi(c_0)$. 
\end{proof}

\section{Compact non-approximable operators between subspaces of $\ell^p$
and $\ell^q$}\label{compactnona}

In this section we first  discuss the existence of compact non-approximable operators between closed subspaces  $X \subset \ell^p$ and $Y \subset \ell^q$ for  $p \neq  q$, 
that is, whether there are such subspaces for which 
\begin{equation}\label{nKA}
\mathcal A(X,Y) \varsubsetneq \mathcal K(X,Y).
\end{equation} 
Recall from Facts \ref{ap1}.(ii) 
that for any $p \neq 2$ there are 
closed subspaces $Z \subset \ell^p$ for which 
$\mathcal A(Z) \varsubsetneq \mathcal K(Z)$. 
However, Theorem \ref{KMJ} implies that \eqref{nKA} cannot hold
for $q < 2 < p$, so the situation becomes subtler  
once $X$ and $Y$ are specified from different classes of spaces. 
The cases $p = 2$ or $q = 2$  are excluded from our discussion
since \eqref{nKA} is impossible in this event. 
 
Our motivation is the quest for additional examples in the setting of Theorem \ref{ideals},
but the question in  \eqref{nKA} has fundamental interest. 
It turns out to involve results  and concepts
which were devised for other purposes. 
In particular, we will use the Banach operator ideal 
of the operators that factor compactly through a subspace of $\ell^r$ as well as 
non-classical approximation properties. We apply these results 
to exhibit several examples related to Theorem \ref{ideals}, including a direct sum 
$X \oplus Y$ in Example \ref{8ideal}, where  $\mathcal K(X \oplus Y)$ contains (at least) 
8 non-trivial closed ideals.

\smallskip

The following Banach factorisation ideals will be essential for our purposes.
Let  $r \in [1,\infty)$ be fixed. For Banach spaces $X$ and $Y$ we define 
\begin{align}\label{rfactor}
\begin{split}
\mathcal{KS}_r(X,Y) = \{T \in \mathcal K(X,Y):\  T & = B  A,\ A \in \mathcal K(X,Z), \  
B \in \mathcal K(Z,Y),\\
& Z \subset \ell^r \textrm{  a closed subspace}\}.
\end{split}
\end{align}
The associated factorisation norm 
of $T \in \mathcal{KS}_r(X,Y) $ is 
\[
\vert T \vert_{\mathcal{KS}_r} = \inf \{ \Vert B\Vert \cdot \Vert A\Vert : T = BA \textrm{ factors as in } \eqref{rfactor}\}
\]
Recall further that the class of classical $r$-compact operators $\mathcal K_r$
is defined by 
\[
\mathcal K_r(X,Y)=\{T\in\mathcal K(X,Y) : T=BA, \textrm{ where }  
A\in\mathcal K(X,\ell^r),\  B\in\mathcal K(\ell^r,Y)\},
\]
and the related factorisation norm is 
\[
\vert T\vert_{\mathcal K_r} = \inf \{||B|| \cdot ||A|| : T= BA, \textrm{ where } 
A\in\mathcal K(X,\ell^r),\  B\in\mathcal K(\ell^r,Y)\}.
\] 
It is known that $(\mathcal K_r,\vert \cdot \vert_{\mathcal K_r})$ is a Banach operator ideal, 
see e.g. \cite[Proposition 1]{J71} and \cite[Theorem 2.1]{FS79}, or \cite[18.3]{Pie}.
The class $(\mathcal{KS}_r,\vert \cdot \vert_{\mathcal{KS}_r})$ 
is also a Banach operator ideal, but this fact is less well-documented in the literature.
For $r = 2$ one has $\mathcal{KS}_2(X,Y)  = \mathcal  K_2(X,Y) \subset \mathcal A(X,Y)$
for all spaces $X$ and $Y$, so this case will not be of interest for us.  

Let $(\mathcal I,\vert \cdot\vert_{\mathcal I})$ 
be a Banach operator ideal.
Recall that the injective hull $\mathcal I^{inj}$ of $\mathcal I$ is 
the Banach operator ideal
\[
\mathcal I^{inj}(X,Y) =\{T\in \mathcal L(X,Y)\mid  J_YT\in \mathcal I(X,\ell^\infty(B_{Y^*})\},
\]
which is normed by $\vert T\vert_{\mathcal I^{inj}} =: \vert J_YT\vert_\mathcal I$
for $T \in \mathcal I^{inj}(X,Y)$. 
Here $J_Y$ is the natural isometric embedding from $Y$ into $\ell^\infty(B_{Y^*})$.
The following result was proved by
Fourie  \cite[Theorem 2.1]{F83} (it is also stated without proof 
on lines 4-6 of \cite[p. 529]{Pie14}).

\begin{prop}\label{Kfactor1} 
If $1 \le r < \infty$, then 
\[
\mathcal{KS}_r=\mathcal  K_r^{inj},
\]
with equality of the respective operator ideal norms.
In particular,  $(\mathcal{KS}_r,\vert \cdot \vert_{\mathcal{KS}_r})$ 
is a Banach operator ideal.
\end{prop}

\begin{remarks}
(i) The argument in \cite{F83} uses the additional fact \cite[Theorem 2.3]{FS79} 
that in the factorisation $T = BA$ from \eqref{rfactor} it suffices to assume that $B$ 
is a bounded operator. We note that it is possible to avoid this step by using 
the compact extension property of $\ell^\infty(I)$. 
There are also direct approaches: 
$(\mathcal{KS}_r,\vert \cdot \vert_{\mathcal{KS}_r})$ 
can be shown a Banach operator ideal either by modifying the argument 
\cite[Hilfssatz 1]{Pie70} for the class of 
operators that factors boundedly through $\ell^r$, or
the argument from \cite[Proposition 1]{J71} for the case $\mathcal K_r(X,Y)$.
We leave the details for the reader.

(ii) There is an explicit characterisation of 
$\mathcal  K_r^{inj}$ in \cite[Theorem 3.6]{F18}. 
\end{remarks}

Our primary interest lies in the closed ideals
$\overline{\mathcal{KS}_r(X)}$, for which
$\mathcal A(X) \subset \overline{\mathcal{KS}_r(X)} \subset \mathcal K(X)$ 
for all Banach spaces $X$. To make our notation less cumbersome 
we have chosen to introduce 
the abbreviation $\mathcal{KS}_r =: \mathcal  K_r^{inj}$.
We note that Proposition \ref{Bfactor}.(ii) below implies that for any 
$r \neq 2$ there is a Banach space $X$ such that
\[
\overline{\mathcal  K_r(X)} \varsubsetneq \mathcal{KS}_r(X) = \mathcal  K_r^{inj}(X).
\]
Moreover, the class $\mathcal{KS}_r$ is not monotone for 
$r \in [1,\infty)$ by Proposition \ref{ideals2}.

\smallskip

We first observe that certain ideal components $\mathcal{KS}_r(X,Y)$  
and $\overline{\mathcal{KS}_r(X,Y)}$ can be identified among the 
closed subspaces of 
$\ell^p$-spaces  
thanks to Theorem \ref{KMJ} and the compact factorisation theorems  of
Figiel \cite{F} and  Bachelis  \cite{B}.

\begin{prop}\label{Bfactor}
Suppose that $p, q \in [1,\infty)$, where $p \neq 2$ and $q \neq 2$, and let $X \subset \ell^p$
and  $Y \subset \ell^q$ be arbitrary closed subspaces. Then the following holds:

\smallskip

(i) $\mathcal K(Z,Y) = \mathcal{KS}_q(Z,Y) = \overline{\mathcal{KS}_q(Z,Y)}$ for any Banach space $Z$, so that  
\[
\mathcal K(Y) = \mathcal{KS}_q(Y) = \overline{\mathcal{KS}_q(Y)}.
\]

(ii) There is a closed subspace $Z \subset \ell^p$ such that 
$\overline{\mathcal  K_p(Z)} \varsubsetneq \mathcal{KS}_p(Z)=\mathcal  K_p^{inj}(Z)$.

(iii) If $p < 2 < q$,  then 
\[
\overline{\mathcal{KS}_q(Z,X)} = \mathcal A(Z,X) \textrm{ and }
\overline{\mathcal{KS}_p(Y,Z)} = \mathcal A(Y,Z)
\]
for any Banach spaces $Z$.
\end{prop}

\begin{proof}
(i) Suppose that $T \in \mathcal K(Z,Y)$ is arbitrary. The compact factorisation results in
\cite[Theorem 7.4]{F} or \cite[Theorem 2']{B} imply that 
there is a closed subspace $W \subset \ell^q$ together with  compact operators
$A \in \mathcal K(Z,W)$, $B \in \mathcal K(W,Y)$ such that  $T = BA$. In other words,
$\mathcal K(Z,Y) = \mathcal{KS}_q(Z,Y)$.

\smallskip

(ii) Recall from Facts \ref{ap1}.(ii) that for all $p \neq 2$ 
there is a closed subspace $Z \subset \ell^p$ such that 
$\mathcal A(Z) \varsubsetneq  \mathcal K(Z)$. Clearly 
$\overline{\mathcal  K_p(Z)} = \mathcal A(Z)$, while 
$\mathcal{KS}_p(Z) = \mathcal K(Z)$ by part (i). 

\smallskip

(iii) If $T \in \mathcal{KS}_q(Z,X)$, then there is  a closed subspace $W \subset \ell^q$
and compact operators $A \in \mathcal K(Z,W)$, $B \in \mathcal K(W,X)$ such that  $T = BA$.
Theorem \ref{KMJ} implies that here $B \in \mathcal A(W,X)$, 
since $p < 2 < q$.
The first equality follows after passing to the uniform closure. The argument 
for the second equality  is similar.
\end{proof}

In Proposition \ref{Bfactor}.(i)  the equality  $\mathcal K(Z,Y) = \mathcal{KS}_q(Z,Y)$
implies that $\vert \cdot \vert_{\mathcal{KS}_q}$
and $\Vert \cdot \Vert$ are equivalent norms on $\mathcal K(Z,Y)$ by the open mapping theorem.

\medskip

We will split our discussion of the examples of strict inclusion \eqref{nKA} 
into the cases $p, q \in [1,2)$ and $p,q \in (2,\infty)$, since they require quite different tools.
We first record for comparison the following (essentially known) 
version of Pitt's theorem for closed subspaces of 
$\ell^r$-spaces. The argument is a straightforward modification of the classical perturbation 
argument for that result, see e.g. \cite[Theorem 2.1.4]{AK} or  \cite[Proposition 2.c.3]{LT1},
and it will not be reproduced here.

\begin{prop}\label{pitt}
Suppose that $1 \le s < r < \infty$, and let $X \subset \ell^r$ and $Y \subset \ell^s$ be 
arbitrary closed subspaces. Then
\[
\mathcal L(X,Y) = \mathcal K(X,Y).
\]
This identity is also valid for closed subspaces $X \subset c_0$.  
\end{prop}

The cases $p, q \in [1,2)$ related to  \eqref{nKA} revisit  
a factorisation theorem of Figiel \cite{F} for 
compact operators mapping into $L^p(\mu)$-spaces.
Let $X \subset \ell^p$ and $Y \subset \ell^q$ be closed subspaces.
The following result demonstrates that 
Theorem \ref{KMJ} fails to hold in the range
$1 \le p < q < 2$, and that Proposition \ref{pitt} 
does not include the class $\mathcal A(X,Y)$ for $1 \le q < p < 2$.
(Analogous remarks also apply to the cases
 $2 < q < p < \infty$ in view of  Theorem \ref{reinov}  below.)
 
 \smallskip

 \begin{thm}\label{Figiel}
(i) If   $1 \le q \le p < 2$, then 
\[
\mathcal{KS}_q(X) = \overline{\mathcal{KS}_q(X)} = \mathcal K(X)
\]
 for any closed  subspace  $X \subset \ell^p$.

(ii)  Suppose that $1 \le p, q < 2$ and $p \neq q$. Then there are closed subspaces
$X \subset \ell^p$ and $Y \subset \ell^q$  for which the strict inclusion \eqref{nKA} holds.

(iii) For $1 \le p < q < 2$ there is a closed subspace $X \subset \ell^p$ such that 
the quotient algebra
\[
\overline{\mathcal{KS}_q(X)}/\mathcal A(X)
\]
is non-nilpotent and infinite-dimensional.
\end{thm}

\begin{proof}
(i)  Let  $1 \le q \le p < 2$ and $X \subset \ell^p$ be a closed subspace. Suppose that 
$T \in \mathcal K(X)$ is arbitrary.  It is a classic fact that 
$\ell^p$ embeds isometrically into $L^q(0,1)$ 
for $1 \le q \le p < 2$, see e.g. \cite[Theorem 6.4.18]{AK}, 
so that  $X$ embeds isometrically into $L^q(0,1)$. 
Hence it follows from \cite[Theorem 7.4]{F} that $T \in  \mathcal K(X)$ factors 
compactly through a  closed subspace of $\ell^q$, that is,
 there is a closed subspace $Z \subset \ell^q$
as well as  compact operators $A \in \mathcal K(X,Z)$ and $B \in \mathcal K(Z,X)$ so that
$T = BA$. In other words,
$\mathcal K(X) = \mathcal{KS}_q(X) = \overline{\mathcal{KS}_q(X)}$. 

\smallskip

(ii)  Suppose that  $1 \le q < p < 2$. 
According to  Facts \ref{ap1}.(ii) there is  a 
closed subspace $X \subset \ell^p$ that carries a compact non-approximable operator 
$T \in  \mathcal K(X) \setminus \mathcal A(X)$. By part (i) we know that 
$T \in \mathcal{KS}_q(X) = \mathcal K(X)$, so there is 
a closed subspace $Z \subset \ell^q$ together with 
$A \in \mathcal K(X,Z)$, $B \in \mathcal K(Z,X)$, 
so that $T = BA$. Here neither $A$ nor $B$ can be  approximable operators, 
so that
\begin{equation}\label{p2}
\mathcal A(X,Z) \varsubsetneq \mathcal K(X,Z)  \textrm{ and } 
\mathcal A(Z,X) \varsubsetneq \mathcal K(Z,X).
\end{equation}
The first strict inclusion gives the claim for $1 \le q < p < 2$, while 
the claim for $1 \le p < q < 2$ follows from the second strict inclusion in 
\eqref{p2} after exchanging the roles of $p$ and $q$. (In fact, part (iii) 
contains a much stronger result for $1 \le p < q < 2$.)

\smallskip

(iii) Let $1 \le p < q < 2$. 
In view of Theorem \ref{nil} there is   a closed subspace
$Z \subset \ell^q$ and a compact operator $T \in \mathcal K(Z)$ 
for which $T^n \notin \mathcal A(Z)$ for any $n \in \mathbb N$. 
Since $1 \le p < q < 2$,  
we know that $T \in \mathcal{KS}_p(Z)$ by part (i), so there is a closed subspace 
$X \subset \ell^p$
and a factorisation  $T = BA$, where $A \in \mathcal K(Z,X)$ and $B \in \mathcal K(X,Z)$. 
Here $AB \in \mathcal {KS}_q(X)$, because 
$AB$ factors compactly through $Z \subset \ell^q$.
Moreover,  $(AB)^n \notin \mathcal A(X)$ for any 
$n \in \mathbb N$, since 
\[
T^{n+1} = (BA)^{n+1} = B(AB)^nA \notin \mathcal A(X).
\]

For complex scalars the conclusion that 
$\overline{\mathcal{KS}_q(X)}/\mathcal A(X)$ is infinite-dimensional 
follows from general  Banach algebra theory. Namely,  
the quotient $\overline{\mathcal{KS}_q(X)}/\mathcal A(X)$ is a 
closed ideal of the radical Banach algebra  $\mathcal K(X)/\mathcal A(X)$, 
so that  $\overline{\mathcal{KS}_q(X)}/\mathcal A(X)$
is also a radical algebra. Moreover, since  $(AB)^n \notin \mathcal A(X)$ for any 
$n \in \mathbb N$,  it follows that $\overline{\mathcal{KS}_q(X)}/\mathcal A(X)$ is non-nilpotent.
We conclude that the radical quotient algebra  $\overline{\mathcal{KS}_q(X)}/\mathcal A(X)$
is infinite-dimensional by  \cite[Proposition 1.5.6(iv)]{D00}.

For real scalars we apply Proposition \ref{real} to 
the closed ideals $\mathcal A(X) \subset  \overline{\mathcal{KS}_q(X)}$ of
$\mathcal K(X)$ and get that 
the real quotient algebra $\overline{\mathcal{KS}_q(X)}/\mathcal A(X)$
is radical in the sense of \eqref{rrad}. 
Since $\overline{\mathcal{KS}_q(X)}/\mathcal A(X)$ 
is non-nilpotent by construction, 
we may also apply \cite[Proposition 1.5.6(iv)]{D00} in the real case to
deduce that $\overline{\mathcal{KS}_q(X)}/\mathcal A(X)$ is infinite-dimensional.
\end{proof}

The uniform closures $\overline{\mathcal{KS}_r(Z)}$ 
provide examples of non-trivial closed ideals of
$\mathcal K(Z)$
for certain Banach spaces $Z$. This also
gives an alternative approach to particular instances of the examples
contained in Theorem \ref{ideals}. 

\begin{prop}\label{ideals2}
Suppose that $1 \le p < 2 < q < \infty$, and use Facts \ref{ap1}.(ii) to pick closed subspaces
$X \subset \ell^p$ and $Y \subset \ell^q$ 
for which $\mathcal A(X) \varsubsetneq \mathcal K(X)$ and
$\mathcal A(Y) \varsubsetneq \mathcal K(Y)$. Then 
\begin{equation}\label{Kr}
\mathcal A(X\oplus Y) \varsubsetneq \overline{\mathcal{KS}_r(X\oplus Y)} \varsubsetneq 
\mathcal K(X\oplus Y)
\end{equation}
for $r = p$ and $r = q$, where $\overline{\mathcal{KS}_p(X\oplus Y)}$ and  
$\overline{\mathcal{KS}_q(X\oplus Y)}$ are incomparable ideals.

In addition, $\mathcal{KS}_p(X\oplus Y)$ and  
$\mathcal{KS}_q(X\oplus Y)$ are also incomparable, so the
classes $\mathcal{KS}_r$ are not monotone for $r \in [1,\infty)$. 
\end{prop}

\begin{proof}
Since $\mathcal{KS}_p$ is a Banach  operator ideal it follows from parts (i) and (iii) of Proposition \ref{Bfactor} that the components of the closure $\overline{\mathcal{KS}_p(X\oplus Y)}$
satisfy
\[
\overline{\mathcal{KS}_p(X\oplus Y)} = 
 \left( \begin{array}{ccc}
\mathcal K(X) &  \mathcal A(Y,X)  \\
\overline{\mathcal{KS}_p(X,Y)}  &  \mathcal A(Y) \\
\end{array} \right).
\]
The assumptions on the diagonal components imply that \eqref{Kr} holds.

For $r = q$ one similarly get from parts (i) and (iii) of Proposition \ref{Bfactor} that
\[
\overline{\mathcal{KS}_q(X\oplus Y)}  = 
\left( \begin{array}{ccc}
\mathcal A(X) &  \mathcal A(Y,X)  \\
 \mathcal K(X,Y)  &  \mathcal K(Y) \\
\end{array} \right).
\]
Our assumptions on $X$ and $Y$ again 
yield that $\overline{\mathcal{KS}_q(X\oplus Y)}$ is a non-trivial closed ideal of 
$\mathcal K(X \oplus Y)$, which cannot be compared to 
$\overline{\mathcal{KS}_p(X\oplus Y)}$.

For $\mathcal{KS}_p(X\oplus Y)$ the respective components satisfy
$\mathcal{KS}_p(X) = \mathcal K(X)$
and $\mathcal{KS}_p(Y) \subset \mathcal A(Y)$ by parts (i) and (iii) of Proposition \ref{Bfactor}.
Similarly, $\mathcal{KS}_q(X) \subset \mathcal A(X)$ and 
$\mathcal{KS}_q(Y) = \mathcal K(Y)$, so the classes $\mathcal{KS}_p(X\oplus Y)$ and  
$\mathcal{KS}_q(X\oplus Y)$ cannot be compared.
\end{proof}

The strict inclusion \eqref{nKA} for the cases  $p,q \in (2,\infty)$ 
involves non-classical approximation properties associated to certain Banach operator ideals.
Let $1 \le p < \infty$ and let  $X$ be a Banach space. We denote by
$\ell^p_s(X)$ the vector space of the strongly $p$-summable sequences in $X$, 
and by $\ell^p_w(X)$ that of the weakly $p$-summable sequences in $X$. Recall that
$\ell^p_s(X)$ and $\ell^p_w(X)$  are Banach spaces equipped with their respective natural  norms,
\begin{align*}
\Vert (x_k) \Vert_{p}  & = \big( \sum_{k=1}^\infty  \Vert x_k\Vert^p\big)^{1/p}, \quad (x_k) \in 
\ell^p_s(X),\\
\Vert (x_k) \Vert_{p,w} & = \Big(\sup_{x^* \in B_{X^*}} \sum_{k=1}^\infty \vert  x^*(x_k)\vert^p\Big)^{1/p}, \quad  (x_k) \in \ell^p_w(X).
\end{align*}

We will require the following Banach operator ideals.
Let $1 \le p < \infty$, and $X, Y$ be Banach spaces. 
Recall from \cite[18.1 and 18.2]{Pie} that $U \in \mathcal L(X,Y)$ is \textit{$p$-nuclear},
denoted  $U \in \mathcal N_p(X,Y)$, if  there is a scalar sequence
$(\sigma_j) \in \ell^p$, a bounded sequence $(x^*_j) \subset X^*$,
and a weakly $p'$-summable sequence $(y_j) \in \ell^{p'}_w(Y)$ such that 
\begin{equation}\label{pn}
Ux = \sum_{j=1}^\infty \sigma_j x_j^*(x)y_j, \quad x \in X.
\end{equation}
Here $p'$ is the dual exponent of $p$. The $p$-nuclear norm is  
\[
\vert U\vert_{\mathcal N_{p}} = \inf 
\{ \Vert (\sigma_j)\Vert_{\ell^p} \cdot  \Vert (x_j^*)\Vert_\infty \cdot \Vert (y_j)\Vert_{p',w}:
\eqref{pn} \textrm{ holds}\}.
\]
It is known that  $(\mathcal N_p,\vert \cdot \vert_{\mathcal N_{p}})$ is a Banach operator ideal,
see \cite[18.1 and 18.2]{Pie}.

Following Persson and Pietsch \cite[section 4]{PP}, the operator 
$U \in \mathcal L(X,Y)$ is \textit{quasi $p$-nuclear}, denoted $U \in \mathcal{QN}_p(X,Y)$, 
if there is a strongly $p$-summable sequence $(x^*_j) \in \ell^p_s(X^*)$ such that 
\begin{equation}\label{qn}
\Vert Ux\Vert \le \big( \sum_{j=1}^\infty  \vert x_j^*(x)\vert^p \big)^{1/p}, \quad  x \in X.
\end{equation}
Then $(\mathcal{QN}_p,\vert \cdot \vert_{\mathcal QN_{p}} )$ is also a Banach operator ideal, where 
\[
\vert U\vert_{\mathcal QN_{p}}  = \inf \{ \Vert (x_j^*) \Vert_{p} : 
\eqref{qn} \textrm { holds for } (x^*_j)\}.
\]
We will require the following facts:  $\mathcal{QN}_p = \mathcal{N}_p^{inj}$,
see  \cite[Satz 39]{PP}, and
$\mathcal{QN}_p \subset \mathcal{QN}_q$
for $p < q$, see \cite[Satz 24]{PP}. Moreover, 
\begin{equation}\label{qnfactor}
\mathcal{QN}_p(X,Y) \subset \mathcal {KS}_p(X,Y)
\end{equation}
for any $1 \le p < \infty$ and any  $X$ and $Y$, see 
the proof of \cite[Lemma 5]{PP}.
Another proof is obtained by  combining Proposition \ref{Kfactor1} with
\cite[Theorem 6]{Pie14} and the inclusion displayed on line 9 of \cite[page 529]{Pie14}.

We say that $T \in \mathcal L(X,Y)$ is a \textit{Sinha-Karn $p$-compact} operator, 
denoted $T \in \mathcal{SK}_p(X,Y)$, if there is 
a strongly $p$-summable sequence $(x_n) \in \ell^p_s(X)$ such that
\[
T(B_X)  \subset \big \{\sum_{n=1}^\infty a_nx_n: (a_n) \in B_{\ell^{p'}}\big \}.
\] 
There is a complete norm $\vert \cdot \vert_{\mathcal{SK}_p}$ on $\mathcal{SK}_p(X,Y)$
such that $(\mathcal{SK}_p,\vert \cdot \vert_{\mathcal{SK}_p})$ is a Banach operator ideal,
see e.g. \cite[Theorem 4.2]{SK} or \cite[Theorem 1]{Pie14}. 
This class was introduced by Sinha and Karn \cite{SK}, and we adopt the 
above terminology and the notation $\mathcal{SK}_p$ in order to distinguish 
it from the (historically earlier) class $\mathcal K_p$ of the $p$-compact operators.  
The relationship between 
 $\mathcal{SK}_p$ and  $\mathcal K_p$ is discussed in \cite[section 3.2]{O12} and 
 \cite[page 529]{Pie14}.
Moreover, it holds that $\mathcal{SK}_p(X,Y) \subset \mathcal K(X,Y)$ 
for any $1 \le p < \infty$ and any spaces $X$ and $Y$,
see e.g. \cite[page 949]{O12}.  
We will also require the following duality relation 
(see \cite[Proposition 3.8]{DPS10} as well as
\cite[Theorem 7]{Pie14} or \cite[Corollary 2.7]{GLT}).

\begin{fact}\label{skp}
Let  $1 \le p < \infty$,  and $X,Y$ be Banach spaces. Then 
\[
T \in \mathcal{SK}_p(X,Y) \textrm{ if and only if }  T^* \in \mathcal{QN}_p(Y^*,X^*).
\]
\end{fact}

Let $(\mathcal I,\vert \cdot \vert_{\mathcal I})$ be a Banach operator ideal  such that
$\mathcal I \subset \mathcal K$,  that is, 
$\mathcal I(X,Y) \subset \mathcal K(X,Y)$ for all spaces $X$ and $Y$.
We will say that the Banach space $X$ has the \textit{uniform $\mathcal I$-approximation property}
(abbreviated uniform $\mathcal I$-A.P.), if 
\begin{equation}\label{uiap}
\mathcal I(Y,X) \subset \overline{\mathcal F(Y,X)} 
 = \mathcal A(Y,X)
\end{equation}
holds for all Banach spaces $Y$.  

Special instances of the uniform $\mathcal I$-A.P. have recently 
been studied (under varying terminology) 
for various Banach operator ideals $\mathcal I$, see e.g. \cite{SK},
\cite{DOPS}, \cite{CK}, \cite{LT} and \cite{Kim19}. For example,
Sinha and Karn \cite{SK} introduced (an equivalent version of) the 
uniform $\mathcal{SK}_p$-A.P. as the $p$-approximation property,
see \cite[Theorem 2.1]{DOPS} for  the
equivalence with \eqref{uiap}.
We have found it convenient to slightly modify the terminology 
suggested by Lassalle and Turco  \cite[page 2460]{LT} 
in order to distinguish the uniform $\mathcal I$-A.P. from the following property 
considered e.g. by Oja \cite{O12}:
the Banach space $X$ is said to have the \textit{$\mathcal I$-approximation property} 
($\mathcal I$-A.P.) if 
\[
\mathcal I(Y,X) = \overline{\mathcal F(Y,X)}^{\vert \cdot \vert_{\mathcal I}} 
\]
holds for all Banach spaces $Y$. Note that if $X$ has the $\mathcal I$-A.P., then $X$ 
also has the uniform $\mathcal I$-A.P., since $\Vert \cdot \Vert \le \vert \cdot \vert_{\mathcal I}$.
The classical A.P. coincides with the uniform 
$\mathcal K$-A.P. as well as the $\mathcal K$-A.P.
We will here (mostly) be concerned with the uniform $\mathcal I$-A.P., since we are 
interested in closed ideals of $\mathcal K(X)$.

\smallskip 

We will require the following auxiliary results, which connect the failure of the 
uniform $\mathcal I$-A.P. for
certain Banach operator ideals $\mathcal I$ 
to the existence of non-approximable operators in $\mathcal {KS}_p$.
In parts (ii) and (iii) the uniform 
$\mathcal{SK}_p$-A.P. is only relevant for 
$2 < p < \infty$, since any Banach space $X$ has this property 
for $1 \le p \le 2$, see \cite[Theorem 6.4]{SK} or \cite[Corollary 2.5]{DOPS}.
Recall for Banach operator ideals $\mathcal I$ and $\mathcal J$ that $\mathcal I
\subset \mathcal J$ means that $\mathcal I(X,Y) \subset \mathcal J(X,Y)$
for all $X, Y$, and $\vert \cdot \vert_{\mathcal J} \le \vert \cdot \vert_{\mathcal I}$.

\begin{lma}\label{kpap}
(i) Let $1 \le p < \infty$ and suppose that the Banach space $X$ 
fails to have the uniform $\mathcal I$-A.P. for some
Banach operator ideal $\mathcal I \subset \mathcal{KS}_p$. Then there is a closed subspace 
$Z \subset \ell^p$ such that 
\[
\mathcal A(Z,X) \varsubsetneq \mathcal K(Z,X).
\]

(ii) Let  $2 < p < \infty$ and suppose that $X$ 
fails to have the uniform $\mathcal{SK}_p$-A.P.
Then there is a closed subspace $Z \subset \ell^p$ such that 
\[
\mathcal A(X^*,Z) \varsubsetneq \mathcal K(X^*,Z).
\]

(iii) Let  $2 < p < \infty$ and suppose that  $X$ 
fails to have the uniform $\mathcal{SK}_p$-A.P.
Then there is a closed subspace $Z \subset \ell^p$ such that 
\[
\mathcal A(X^{**},Z) \varsubsetneq \mathcal K(X^{**},Z).
\]
If  $X$ is reflexive, then $\mathcal A(X,Z) \varsubsetneq \mathcal K(X,Z)$.
\end{lma}

\begin{proof}
(i)  From the failure of \eqref{uiap} there is a Banach space $Y$ and an operator 
$U \in \mathcal I(Y,X)$ such that 
$U \notin \mathcal A(Y,X)$. Since 
$U \in \mathcal I(Y,X) \subset \mathcal{KS}_p(Y,X)$ by assumption,
there is a closed subspace $Z \subset \ell^p$ and a compact factorisation $U = BA$,
where $A \in \mathcal K(Y,Z)$ and $B \in \mathcal K(Z,X)$. 
Here  $B \notin \mathcal A(Z,X)$.

\smallskip

(ii) If $X$ fails to have the uniform $\mathcal{SK}_p$-A.P., then from
\eqref{uiap} there is a Banach space $Y$ and 
an operator $T \in \mathcal{SK}_p(Y,X) \setminus \mathcal A(Y,X)$.
From Fact \ref{skp}  and the inclusion \eqref{qnfactor} we have
\[
T^* \in \mathcal{QN}_p(X^*,Y^*) \subset \mathcal{KS}_p(X^*,Y^*).
\]
Hence there is a closed subspace $Z \subset \ell^p$ and a compact factorisation 
$T^* = BA$ through $Z$.  Note that 
$T^*$ cannot be approximable. Namely, if $T^*$ were approximable $X^{*} \to Y^*$,
then  the principle of local reflexivity implies that 
$T$ is also approximable, see e.g. \cite[Propositions 2.5.1-2.5.2]{CS} or 
\cite[Theorem 11.7.4]{Pie}. We get that 
$\mathcal A(X^*,Z) \varsubsetneq \mathcal K(X^*,Z)$, 
since $A \in \mathcal K(X^*,Z)$ cannot be an approximable operator.

(iii)  If $X$ fails to have the uniform $\mathcal{SK}_p$-A.P., then
$X^*$ also fails  this property by 
the duality result in \cite[Theorem 2.7]{CK}.
Part (ii) implies that there is a closed subspace $Z \subset \ell^p$ such that 
$\mathcal A(X^{**},Z) \varsubsetneq \mathcal K(X^{**},Z).$
\end{proof}

The desired examples of a strict inclusion  \eqref{nKA} for $p,q \in (2,\infty)$ 
revisits an intricate construction  of Reinov \cite{Rei82}
concerning the failure of duality for $p$-nuclear operators, 
which we reinterpret in terms of the uniform $\mathcal{QN}_p$-A.P.
Reinov's argument in part (i) of \cite[Lemma 1.1 and Corollary 1.1]{Rei82} 
covers the case $p = q \in (2,\infty)$, which extends by monotonicity 
to $2 < q < p < \infty$ (see the argument of Theorem \ref{reinov}).  
For $2 < p < q < \infty$ 
we somewhat modify Reinov's construction, and in the interest of readability 
we have added details to the condensed explanation in \cite{Rei82}. 
Recall that $\mathcal{QN}_p \subset \mathcal A$ for $1 \le p \le 2$
by monotonicity and \cite[18.1.8 and 18.1.4]{Pie}, 
so any Banach space $X$ has the uniform $\mathcal{QN}_p$-A.P. for $p \in [1,2]$.

\begin{thm}\label{reinov}
Let $2<p, q<\infty$ and $p \neq q$. Then there exists a closed subspace 
$Y\subset\ell^q$ such that $Y$ fails to have the uniform $\mathcal{QN}_p$-A.P.,
and a closed subspace $X \subset \ell^p$ for which
\[
\mathcal A(X,Y) \varsubsetneq \mathcal K(X,Y).
\]
\end{thm}

\begin{proof}
We first consider the details for $2 < p < q < \infty$, where  
we will exhibit (by closely following Reinov's outline) 
a closed subspace $Y\subset\ell^q$ together with 
an operator 
\[
U\in \mathcal{QN}_p(Z,Y)\setminus \mathcal A(Z,Y)
\]
for some Banach space $Z$. 
The second claim follows immediately from Lemma \ref{kpap}.(i), since 
$\mathcal{QN}_p \subset \mathcal{KS}_p$ by \eqref{qnfactor}. Hence
the compact factorisation $U = BA$ through some closed subspace 
$X \subset \ell^p$ provides an operator 
$B \in \mathcal K(X,Y) \setminus \mathcal A(X,Y)$.

\smallskip

The starting point is the construction of Davie \cite{Da73} (see also \cite[Section 2.d]{LT1}):
There is a matrix $A=(a_{ij})_{i,j=1}^\infty$ with the following properties: 
\begin{enumerate}[(i)]
    \item $A^2=0$, 
    \item tr $A=\sum_{k=1}^\infty a_{kk}=1$ and 
    \item $\sum_{n=1}^\infty \lambda_n^\alpha<\infty$ for all $\alpha>2/3$, where 
    $\lambda_n=\sup_{k\in\mathbb N}|a_{nk}|>0$ for all $n\in\mathbb N$.
\end{enumerate}
The matrix $A$ defines a $1$-nuclear operator  
$\ell^1\to\ell^1$ through 
\[Ax=(\sum_{n=1}^\infty a_{kn}x_n)_{k=1}^\infty,\quad x=(x_n)\in \ell^1.\]
Define bounded operators $V:\ell^1\to\ell^\infty$ and $\Delta:\ell^\infty\to\ell^1$ as follows:
\[
Vx=(\lambda_{k}^{-1}\sum_{n=1}^\infty a_{kn}x_n)_{k=1}^\infty, \quad x=(x_n)\in \ell^1,\
\textrm{ and }\ 
\Delta x=(\lambda_kx_k),\quad x=(x_k)\in\ell^\infty.
\] 
Clearly $A=\Delta V$, where $\Delta=\sum_{n=1}^\infty \lambda_ne_n\otimes e_n$ is a $1$-nuclear diagonal operator. Here $(e_n)$ denotes the unit vector basis in $\ell^1$ (and
subsequently also  in  $\ell^r$ for appropriate $r$). 
It follows from (i) and (ii) that
\begin{enumerate}
\item[(iv)]    $(\Delta V)^2=0$ and 
\item[(v)] tr$(\Delta V)=\sum_{k=1}^\infty \langle \Delta Ve_k,e_k\rangle = 
\sum_{k=1}^\infty a_{kk} =1$.
\end{enumerate} 

Let $\frac{2}{3}<\alpha<1-\frac{2}{3p}$ (this is possible because $p > 2$)
and put $\lambda_n^{(1)}=\lambda_n^{1-\alpha}$ and $\lambda_n^{(2)}=\lambda_n^\alpha$ 
for all $n\in\mathbb N$. It follows from  (iii) 
that $(\lambda_n^{(1)})\in\ell^p$ and
$(\lambda_n^{(2)})\in\ell^1$. 
Consider the following commuting diagram:

\begin{center}
$\xymatrix{
\operatorname{\ell^1}\ar[rr]^V\ar[d]^{V_1}&&\operatorname{\ell^\infty}\ar[rr]^\Delta\ar[d]^{\Delta_1}&&\operatorname{\ell^1}&\\
\operatorname{\emph Z}\ar[r]^{\tilde V}\ar[rrd]^U&\operatorname{\emph Y_0}\ar[ru]^{j_1}\ar[rd]^{\tilde\Delta_1}&\operatorname{\ell^\emph p}\ar[r]^{i}&\operatorname{\ell^\emph q}\ar[ru]^{\Delta_2}\\
&&\operatorname{\emph Y}\ar[ru]^{j_2}
}
$
\end{center}

\noindent Here $Z=\ell^1/\ker V$ and $Y_0=\overline{V\ell^1} \subset \ell^\infty$; 
the operators $V_1$ and $\tilde V$ are induced by $V$ and $j_1$ is the isometric inclusion
map; 
\[
\Delta_1=\sum_{k=1}^\infty\lambda_k^{(1)}e_k\otimes e_k \textrm{ and } 
\Delta_2=\sum_{k=1}^\infty\lambda_k^{(2)}e_k\otimes e_k,
\]
and $i$ is the canonical inclusion; $Y=\overline{i\Delta_1 j_1 Y_0} \subset \ell^q$ and $\tilde\Delta_1=j_2^{-1}i\Delta_1j_1$,
where $j_2$ is the isometric inclusion map and  $U=\tilde\Delta_1\tilde V$. 

Clearly $\Delta_1\in \mathcal N_p(\ell^\infty,\ell^p)\subset \mathcal{QN}_p(\ell^\infty,\ell^p)$, and it is easy to check by using \eqref{qn}
that the restriction $\tilde\Delta_1\in \mathcal{QN}_p(Y_0,Y)$. Deduce that 
$U =\tilde\Delta_1\tilde V \in \mathcal{QN}_p(Z,Y)$. 
Next we verify that $U$ is not an approximable operator. Towards this 
define the bounded functional $\Phi \in \mathcal L(Z,Y)^*$ by
\[
\Phi(S) =  \sum_{k=1}^\infty \lambda_k^{(2)}\langle SV_1e_k,j_2^*e_k\rangle 
\textrm{ for } S \in \mathcal L(Z,Y).
\]

\textit{Claim:}   $\Phi(T)=0$ for all $T\in \mathcal A(Z,Y)$. 

\smallskip

By linearity and continuity it will be enough to show that 
$\Phi(z^*\otimes y)=0$ for all $z^*\in Z^*$ and $y\in Y$. Moreover, by the definition of $Y$ 
we may also (by continuity) assume that $y\in j_2^{-1}i\Delta_1V\ell^1$. 
Thus there is $a\in\ell^1$ such that $y=j_2^{-1}i\Delta_1 Va$.  We get that
\begin{align*}
\Phi(z^*&\otimes y)=\sum_{k=1}^\infty \lambda_k^{(2)}z^*(V_1e_k)\langle y,j_2^*e_k\rangle=\sum_{k=1}^\infty \lambda_k^{(2)}z^*(V_1e_k)\langle i\Delta_1 Va,e_k\rangle\\
&=z^*\Big(V_1\big(\sum_{k=1}^\infty \langle\lambda_k^{(2)}i\Delta_1Va,e_k\rangle e_k\big)\Big)=z^*(V_1\Delta_2i\Delta_1 Va)=z^*(V_1\Delta Va)=0.
\end{align*}
The last equality follows from the fact that $V_1\Delta V=0$. Namely, by (iv) 
we have $0 = \Delta V \Delta V  = \Delta j_1 \tilde V V_1\Delta V$, where 
$\Delta j_1 \tilde V$ is injective.
Thus $\Phi(T)=0$ for all $T\in \mathcal A(Z,Y)$. On the other hand,
condition (v) implies that
\[
\Phi(U)=\sum_{k=1}^\infty \lambda_k^{(2)}\langle UV_1e_k,j_2^*e_k\rangle=\sum_{k=1}^\infty \langle \lambda_k^{(2)}j_2UV_1e_k,e_k\rangle=\sum_{k=1}^\infty\langle \Delta Ve_k,e_k\rangle=1.
\]
Hence $U\notin \mathcal A(E,F)$, while $U \in \mathcal{QN}_p(Z,Y)$.

\smallskip

Suppose next that  $2 < q < p < \infty$. According to part (i) of 
\cite[Lemma 1.1 and Corollary 1.1]{Rei82} 
there is a closed subspace $Y \subset \ell^q$ together with an operator
$U \in \mathcal{QN}_q(Z,Y) \setminus \mathcal A(Z,Y)$, 
where $Z$ is a suitable Banach space.
Since $q < p$ we get from the monotonicity of the classes 
$\mathcal{QN}_r$ and \eqref{qnfactor}  that 
\[
U \in \mathcal{QN}_q(Z,Y) \subset \mathcal{QN}_p(Z,Y) \subset \mathcal{KS}_p(Z,Y).
\]
Thus $Y$ fails the uniform $\mathcal{QN}_p$-A.P.
Moreover, there is a closed subspace $X \subset \ell^p$ and a compact factorisation
$U = BA$ through $X$. In particular, $B \in \mathcal K(X,Y) \setminus \mathcal A(X,Y)$
is the desired operator.
\end{proof} 

\begin{remarks} 
(i) The diagram of Reinov \cite[page 127]{Rei82}  for the cases
$p = q \in (2,\infty)$ is similar to the one displayed above.  
Our diagram adds the inclusion map $i: \ell^p \to \ell^q$,
but removes the DFJP-factorisation of $V$ through a reflexive space, which is not relevant
for our purposes.  The argument of Reinov produces a closed subspace 
$Y\subset\ell^p$ such that $Y$ fails the uniform $\mathcal{QN}_p$-A.P.,
which strengthens Facts \ref{ap1}.(i) for $p > 2$ (see also Example \ref{kim} below).

\smallskip
 
(ii) It is possible to recover part of Theorem \ref{reinov} 
for certain combinations of $p, q \in (2,\infty)$ 
from an example of Choi and Kim \cite{CK} concerning the uniform 
$\mathcal {SK}_p$-A.P. (Their example is also based on 
Davie's construction  \cite{Da73}.) In fact, for $2 < p < \infty$ and $q > \frac{2p}{p-2}$
there is  in view of \cite[Corollary 2.9]{CK} a closed subspace $X \subset \ell^q$ 
that fails the uniform $\mathcal{SK}_p$-A.P. 
Hence Lemma \ref{kpap}.(iii)  provides 
a closed subspace $Y \subset \ell^p$ for which 
$\mathcal A(X,Y) \varsubsetneq \mathcal K(X,Y)$.
\end{remarks}

\medskip

After these preparations we return to the setting of  Theorem  \ref{ideals}
and the question of how many non-trivial closed ideals we may find
in $\mathfrak{A}_Z$ for spaces $Z$ belonging to that class of direct sums. 
We  proceed step by step with a series of examples,
and first display  a direct sum $Z = X \oplus Y$
for which $\mathfrak{A}_Z$ has at least three closed ideals.

\begin{ex}\label{ideals3} 
Suppose that $2 < p < \infty$ and $q > 2p/(p-2)$. Then there are closed subspaces 
$X \subset \ell^p$ and  $Y\subset \ell^q$ such that the quotient algebra 
$\mathfrak A_{X^*\oplus Y}$ contains (at least) three non-trivial closed ideals. 
\end{ex}

\begin{proof}
From \cite[Corollary 2.9]{CK} we find a closed subspace 
$X_0\subset\ell^q$ that fails the uniform $\mathcal{SK}_p$-A.P. 
According to part (ii) of Lemma \ref{kpap} there is a closed subspace  $Y_0\subset\ell^p$
together with an operator 
\begin{equation}\label{S}
S\in\mathcal K(X_0^*,Y_0)\setminus\mathcal A(X^*_0,Y_0).
\end{equation}
Next we use Facts \ref{ap1}.(ii) to pick closed subspaces 
$M\subset\ell^q$ and $N \subset \ell^p$ such that $\mathfrak A_M\neq\{0\}$
and $\mathfrak A_N\neq\{0\}$. Put
\[
X=X_0\oplus M\text{ and } Y=Y_0\oplus N.
\] 
Thus $X$ is a closed subspace
of $\ell^q$, where $q > 2$, and $Y$ a closed subspace
of $\ell^p$, so that both $X$ and $Y$ have type 2. 
It follows from known facts about type and cotype
that $X^* = X_0^* \oplus M^*$ has cotype 2,  see e.g. \cite[Proposition 11.10]{DJT}. 
Since $\mathfrak A_M\neq\{0\}$ and  $\mathfrak A_N\neq\{0\}$, the quotient algebras 
$\mathfrak A_{X}$ and $\mathfrak A_{Y}$  are both non-trivial by 
\cite[Proposition 4.2]{TW}. Finally, since $\mathfrak A_{X}$ embeds isometrically 
into $\mathfrak A_{X^*}$ by duality  \cite[Proposition 4.1]{TW}, we also know that
$\mathfrak A_{X^*}\neq\{0\}$.
Altogether $X^*$ and $Y$ satisfy the conditions of Theorem  \ref{ideals}. 

 Let $Z = X^* \oplus Y$, and let $\mathcal I$ and $\mathcal J$ be the closed ideals 
of $\mathcal K(Z)$ obtained in Theorem  \ref{ideals}, for which 
$q(\mathcal I)$ and $q(\mathcal J)$ define two non-trivial incomparable closed ideals 
of $\mathfrak A_{Z}$.

\medskip

\textit{Claim}.  $\mathcal I \cap \mathcal J = \left( \begin{array}{ccc}
\mathcal A(X^*) & \mathcal A(Y,X^*) \\
\mathcal K(X^*,Y)  &  \mathcal A(Y) \\
\end{array} \right) $ is a third non-trivial ideal of 
$\mathcal K(X^* \oplus Y)$.

\medskip

In fact,  $\mathcal I \cap \mathcal J \varsubsetneq \mathcal I$  and 
$\mathcal I \cap \mathcal J \varsubsetneq \mathcal J$, since 
$\mathcal A(X^*)  \varsubsetneq \mathcal K(X^*)$ and 
$\mathcal A(Y)  \varsubsetneq \mathcal K(Y)$ by construction. 
Pick $S\in\mathcal K(X_0^*,Y_0)\setminus\mathcal A(X^*_0,Y_0)$ by (\ref{S}), 
and let $\tilde S=:JSP: X^* \to Y$, where
$P:X^*\to X_0^*$ is the natural projection and $J:Y_0\to Y$ is the inclusion map.
Hence $\tilde S \in \mathcal K(X^*,Y) \setminus\mathcal A(X^*,Y)$, and 
\[
\left( \begin{array}{ccc}
0 & 0\\
\tilde S & 0\\
\end{array} \right)
\in \big( \mathcal I \cap \mathcal J\big)  \setminus \mathcal A(X^* \oplus Y).
\]
Consequently, by Proposition \ref{quotient} 
we get the  non-trivial closed ideal  $q(\mathcal I)\cap q(\mathcal J)$  of 
$\mathfrak A_{Z}$, which differs from both $q(\mathcal I)$ and $q(\mathcal J)$. 

We note in passing that $q(\mathcal I)\cap q(\mathcal J)$ is a nilpotent ideal of  
$\mathfrak A_{Z}$, that is, 
if $U, V \in \mathcal I \cap \mathcal J$, then 
$UV \in \mathcal A(X \oplus Y)$. In fact, by \eqref{prod} the component  
\[
(UV)_{2,1} = U_{21}V_ {11} + U_{22}V_{21} \in \mathcal A(X,Y),
\]
since $V_{11}$ and $U_{22}$ are approximable operators. 
\end{proof}

Let $Z = \ell^p$, where $1 \le p < \infty$ and $ p \neq 2$, or $Z = c_0$.
 A classical result of Gohberg, Markus and Feldman 
says that   $\mathcal K(Z)$ is the unique non-trivial closed ideal in $\mathcal L(Z)$, see
e.g. \cite[Section 5.2]{Pie}. It is relevant to ask whether there are 
non-trivial closed ideals
\[
\mathcal A(X) \varsubsetneq \mathcal J \varsubsetneq  \mathcal K(X)
\]
among the closed subspaces $X \subset Z$. 
Our next two results demonstrate  that this is indeed the case  (at least) for $p > 4$
and for $c_0$. For subspaces of  $\ell^p$ this  is based on 
\cite[Corollary 2.9]{CK} and Theorem \ref{nil}, as well as 
properties of quasi $p$-nuclear operators. We stress that the resulting direct sums  
do \textit{not} belong to the class of spaces in Theorem \ref{ideals}.

\begin{ex}\label{newideal}
Let $p>2$ and $q>2p/(p-2)$. By combining \cite[Corollary 2.9]{CK}  and 
\cite[Theorem 2.7]{CK}, there is a closed subspace $X\subset\ell^q$ such that $X^*$ fails
the uniform  $\mathcal{SK}_p$-A.P. By definition there is an operator 
$T\in\mathcal{SK}_p(X_0,X^*)\setminus\mathcal A(X_0,X^*)$ for some Banach space $X_0$. Consequently $T^*\in\mathcal{QN}_p(X,X_0^*)\setminus\mathcal A(X,X_0^*)$  
by Fact \ref{skp} and the reflexivity of $X$. 
 According to the proof of \cite[Lemma 5]{PP}  there is a factorisation $T^*=BA$ through a closed subspace $Y\subset\ell^p$ 
such that $A\in\mathcal{QN}_p(X,Y)$. Moreover, by Theorem \ref{nil} 
there is a closed subspace $Z\subset\ell^p$ together with a compact operator 
$S\in\mathcal K(Z)$ such that 
\begin{equation}\label{examplereference}
S^n\notin\mathcal A(Z)\text{ for all }n\in\mathbb N.
\end{equation}
We consider the closed subspace $W=X\oplus Y\oplus Z$ of $\ell^q\oplus\ell^p$
and claim that
\begin{equation}\label{ideal}
\mathcal A(W)\subsetneq \overline{\mathcal{QN}_p(W)}\subsetneq \mathcal K(W),
\end{equation}
\begin{equation}\label{dideal}
\mathcal A(W^*)\subsetneq \overline{\mathcal{SK}_p(W^*)}\subsetneq \mathcal K(W^*).\end{equation}

In particular, for $p=q>4$ there is a closed subspace $W\subset\ell^p$ 
such that (\ref{ideal}) holds.
\end{ex}

\begin{proof}
First note that $A$ is not approximable, since otherwise $T^*$ would also be approximable. 
Thus $A\in\overline{\mathcal{QN}_p(X,Y)}\setminus \mathcal A(X,Y)$, which implies 
for $W=X\oplus Y\oplus Z$ that
\[
\mathcal A(W)\subsetneq\overline{\mathcal{QN}_p(W)}.
\]  
We claim that $S\in\mathcal K(Z)\setminus\overline{\mathcal{QN}_p(Z)}$ 
(which immediately yields that $\overline{\mathcal{QN}_p(W)}\subsetneq \mathcal K(W)$).
Suppose to the contrary that $S\in\overline{\mathcal{QN}_p(Z)}$ and let 
$(R_n)\subset\mathcal{QN}_p(Z)$ be a sequence such that 
\begin{equation}\label{sequence}
||R_n-S||\to 0\text{ as }n\to\infty.
\end{equation} 
Recall from \cite[section 7]{PP} that $UV \in \mathcal{QN}_r$ 
whenever $U \in \mathcal{QN}_s$ and 
$V \in \mathcal{QN}_t$ are compatible operators for which $1/r=1/s+1/t\le 1$. Fix an integer $m$ such that $p/2\le m<p$. 
By iterating the above product formula 
we get that $R_n^m\in\mathcal{QN}_{p/m}(Z)$ for all $n\in\mathbb N$.
Since 
\begin{equation}\label{qnp}
\mathcal{QN}_{p/m}\subset\mathcal{QN}_2= \mathcal N_2\subset\mathcal A
\end{equation}
by monotonicity and \cite[18.1.8 and 18.1.4]{Pie}, 
the operators $R_n^m\in\mathcal A(Z)$ for all $n\in\mathbb N$.
Conclude from (\ref{sequence}) that $S^m\in\mathcal A(Z)$, which contradicts
(\ref{examplereference}). Thus $S\notin\overline{\mathcal{QN}_p(Z)}$.

The strict inclusions in \eqref{dideal} follow by duality.
In fact, since $W$ is reflexive,  Fact \ref{skp} implies 
that $T \in \overline{\mathcal{QN}_p(W)}$ if and only if 
$T^* \in \overline{\mathcal{SK}_p(W^*)}$.
Finally, note that it is possible to choose $p = q$ for $p > 4$ and $q > 2p/(p-2)$, 
in which case  $W$ is a closed subspace of $\ell^p$. 
\end{proof}

In the analogous example of a closed subspace $X \subset c_0$ we may simultaneously 
involve the closures 
$\overline{\mathcal{QN}_p(X)}$  and $\overline{\mathcal{SK}_p(X)}$, and the
details are somewhat different. 

\begin{ex}\label{c0v2}
Let $2<p<\infty$. Then there is a closed subspace $X\subset c_0$ such that
\[\mathcal A(X)\subsetneq \overline{\mathcal{SK}_p(X)}\subsetneq \mathcal K(X)\text{ and }\mathcal A(X)\subsetneq \overline{\mathcal{QN}_p(X)}\subsetneq\mathcal K(X),
\]
where $\overline{\mathcal{SK}_p(X)}$ and $\overline{\mathcal{QN}_p(X)}$ are incomparable closed ideals.
\end{ex}
\begin{proof} We establish the claim in two parts.

\smallskip

\emph{Claim 1.} There are closed subspaces $M_0$ and $M_1$ of $c_0$ such that \begin{equation}\label{Henriklabel1}
\overline{\mathcal{SK}_p(M_0,M_1)}\not\subset \overline{\mathcal{QN}_p(M_0,M_1)}.\end{equation}

\emph{Claim 2.} There are closed subspaces $M_2$ and $M_3$ of $c_0$ such that \begin{equation}\label{Henriklabel2}
\overline{\mathcal{QN}_p(M_2,M_3)}\not\subset \overline{\mathcal{SK}_p(M_2,M_3)}.\end{equation}

\smallskip

We may then take
\[
X=M_0\oplus M_1\oplus M_2\oplus M_3\subset c_0.
\]
Denote $\mathcal I=\overline{\mathcal{SK}_p(X)}$ and $\mathcal J=\overline{\mathcal{QN}_p(X)}$.
Then (\ref{Henriklabel1}) implies that $\mathcal I\not\subset\mathcal J$ and (\ref{Henriklabel2}) implies that $\mathcal J\not\subset\mathcal I$. Hence $\mathcal I$ and $\mathcal J$ are incomparable closed ideals. Since $\mathcal I$ and $\mathcal J$ 
lie between $\mathcal A(X)$ and $\mathcal K(X)$, we get the strict inclusions 
$\mathcal A(X)\subsetneq\mathcal I\subsetneq\mathcal K(X)$ 
and  $\mathcal A(X)\subsetneq \mathcal J\subsetneq\mathcal K(X)$.

\smallskip

\emph{Proof of Claim 1.} Let $F$ be a reflexive Banach space with type 2 that fails the uniform 
$\mathcal{SK}_p$-A.P. \cite[Corollary 2.9]{CK}. In view of 
\cite[Theorem 2.7]{CK} the dual $F^*$ also fails the uniform $\mathcal{SK}_p$-A.P. 
By definition there is an operator $T\in\mathcal{SK}_p(E,F^*)\setminus \mathcal A(E,F^*)$ 
for a suitable Banach space $E$. 

According to \cite[Proposition 2.9]{GLT} we can factor
$T=VRV_0$ where $V_0$ and $V$ are compact operators and $R \in \mathcal{SK}_p$.
Moreover, $V_0$ factors compactly through a reflexive space $W$ by the Figiel-Johnson factorisation result \cite[Proposition 3.1]{F} or the DFJP-factorisation 
 (see e.g. \cite[Theorem 2.g.11]{LT2} or \cite[Theorem 3.2.1]{GMA}).
Let $V_0=BA$ be the corresponding factorisation, so that $T=VRBA$. We claim that\begin{equation}\label{Henriklabel3}
S=:VRB\in\mathcal{SK}_p(W,F^*)\setminus\overline{\mathcal{QN}_p(W,F^*)}.
\end{equation}
Firstly, $S\in\mathcal{SK}_p(W,F^*)$ because $R \in \mathcal{SK}_p$.
Secondly, in view of  inclusion (3.6) and Example \ref{kim}.(ii) 
below for the cotype 2 space $F^*$, we obtain that
\[
\mathcal{QN}_p(W,F^*)\subset \mathcal{KS}_p(W,F^*)\subset\mathcal A(W,F^*).
\]
Thus if $S\in \overline{\mathcal{QN}_p(W,F^*)}$, then $S$ must be approximable. 
But this would mean that $T=SA$ is approximable,
which is a contradiction. Hence $S\notin \overline{\mathcal{QN}_p(W,F^*)}$

Next we use Terzio\u{g}lu's factorisation result \cite{T} to obtain a compact factorisation $B=B_2B_1$ through a suitable closed subspace $M_0\subset c_0$. 
Similarly, $V$ has a compact factorisation $V=V_2V_1$ through a closed subspace 
$M_1\subset c_0$. Consider $U=:V_1RB_2\in\mathcal{SK}_p(M_0,M_1)$, so that 
\[
S= VRB = V_2V_1RB_2B_1=V_2UB_1.
\] 
From (\ref{Henriklabel3}) we get  that 
$U\in\mathcal{SK}_p(M_0,M_1)\setminus  \overline{\mathcal{QN}_p(M_0,M_1)}$,
 which proves Claim 1.

\smallskip

\emph{Proof of Claim 2.} By using the  reflexivity of $W$ and $F^*$ and the duality in 
Fact \ref{skp} together with (\ref{Henriklabel3}) we obtain that 
\begin{equation}\label{Henriklabel5}
S^*\in \mathcal{QN}_p(F,W^*)\setminus \overline{\mathcal{SK}_p(F,W^*)}.
\end{equation}
According to \cite[Lemma 5]{PP} the operator  $S^*$  factors as $S^*=QT_0P$,
where $P$ and $Q$ are compact operators and $T_0$ is also quasi $p$-nuclear.
As above, by Terzio\u{g}lu's result we can further factor $P$ and $Q$ compactly through closed subspaces $M_2\subset c_0$ and $M_3\subset c_0$, respectively. Let $P=P_2P_1$ and $Q=Q_2Q_1$ be the corresponding factorisations, so that $S^*=Q_2Q_1T_0P_2P_1$. 
Consider  $R=:Q_1T_0P_2\in\mathcal{QN}_p(M_2,M_3)$. 
From (\ref{Henriklabel5}) we obtain that
\[
R\in\mathcal{QN}_p(M_2,M_3)\setminus \overline{\mathcal{SK}_p(M_2,M_3)},
\]
which establishes Claim 2.
\end{proof}

We are now in  position to combine the spaces appearing in 
Examples \ref{ideals3} and \ref{newideal}, 
and to obtain a particular direct sum $X \oplus Y$ in the setting of Theorem \ref{ideals},
for which the quotient algebra $\mathfrak{A}_{X \oplus Y}$ contains
(at least) 8 non-trivial closed ideals.

\begin{ex}\label{8ideal}
Suppose that $p > 2$ and $q > 2p/(p-2)$. Let $W \subset \ell^p \oplus \ell^q$ be the 
closed subspace constructed in  Example \ref{newideal},
for which the strict inclusions \eqref{ideal} and \eqref{dideal} hold. 
By arguing as at the beginning of Example \ref{ideals3},
let $X_0 \subset \ell^q$ and $Y_0 \subset \ell^p$ be closed subspaces
such that $\mathcal A(X_0^*,Y_0) \varsubsetneq 
\mathcal K(X_0^*,Y_0)$.
Consider the direct sum 
\[
X \oplus Y =: (W^* \oplus X_0^*) \oplus (W \oplus Y_0).
\]

It is easy to check that the direct sums $W \oplus X_0$ and $Y = W \oplus Y_0$ have type 2, 
so that $X = W^* \oplus X_0^*$ has cotype 2 by 
\cite[Proposition 11.10]{DJT}.  Hence $X \oplus Y$ belongs to  the 
class of spaces from Theorem \ref{ideals}, for which $\mathcal K(Y,X) = \mathcal A(Y,X)$
by Theorem \ref{KMJ}.  By using suitable component operators  as before
one verifies the following strict inclusions by using \eqref{ideal}, \eqref{dideal} 
and the fact that $\mathcal A(X_0^*,Y_0) \varsubsetneq \mathcal K(X_0^*,Y_0)$:

\smallskip

(i) $\mathcal A(X) \subsetneq \mathcal I_0 =: \overline{\mathcal{SK}_p(X)}\subsetneq 
\mathcal K(X)$,

(ii)  $\mathcal A(Y) \subsetneq \mathcal J_0 =: \overline{\mathcal{QN}_p(Y)}\subsetneq 
\mathcal K(Y)$, and

(iii) $\mathcal A(X,Y) \varsubsetneq \mathcal K(X,Y)$.

\smallskip

\noindent We claim that $\mathcal K(X \oplus Y)$ contains (at least) 
the following 8 non-trivial 
closed ideals, where $\mathcal I$ and $\mathcal J$ are the 
ideals constructed in Theorem \ref{ideals}:
\begin{align*}
\mathcal I & = \left( \begin{array}{ccc}
\mathcal K(X) & \mathcal A(Y,X) \\
\mathcal K(X,Y)  &  \mathcal A(Y) \\
\end{array} \right), \ 
\mathcal J =
\left( \begin{array}{ccc}
\mathcal A(X) & \mathcal A(Y,X) \\
\mathcal K(X,Y)  &  \mathcal K(Y) \\
\end{array} \right), \ 
\mathcal I \cap \mathcal J, \\   
\mathcal J_1 & = \left( \begin{array}{ccc}
\mathcal I_0 & \mathcal A(Y,X) \\
\mathcal K(X,Y)  &  \mathcal K(Y)\\
\end{array} \right), \
\mathcal I_1 = \left( \begin{array}{ccc}
\mathcal  K(X) & \mathcal A(Y,X) \\
\mathcal K(X,Y)  &  \mathcal J_0\\
\end{array} \right),  \ 
\mathcal I_1 \cap \mathcal J_1, \\
\mathcal I_2 & = \left( \begin{array}{ccc}
\mathcal I_0 & \mathcal A(Y,X) \\
\mathcal K(X,Y)  &  \mathcal A(Y)\\
\end{array} \right)  \textrm{ and } 
\mathcal J_2 = \left( \begin{array}{ccc}
\mathcal A(X) & \mathcal A(Y,X) \\
\mathcal K(X,Y)  &  \mathcal J_0\\
\end{array} \right).
\end{align*}
It is a straightforward exercise to verify from  \eqref{prod} 
that the new classes above are closed ideals. 
We consider $\mathcal I_2$ as an example, and leave  
the cases $\mathcal I_1$, $\mathcal J_1$ and $\mathcal J_2$ for the reader. 

Suppose first that $U \in \mathcal I_2$ and $V \in \mathcal K(X \oplus Y)$. 
It follows that the component
\[
(UV)_{1,1} = U_{11}V_ {11} + U_{12}V_{21} \in \mathcal I_0
\]
since $U_{11} \in \mathcal I_0$ , while
$(UV)_{2,2} = U_{21}V_{12} + U_{22}V_{22} \in \mathcal A(Y)$ since
$V_{12}, U_{22} \in \mathcal A$. 
Suppose next that $U \in \mathcal K(X \oplus Y)$ and $V \in \mathcal I_2$.
In this case 
\[
(UV)_{1,1} = U_{11}V_ {11} + U_{12}V_{21} \in \mathcal I_0
\]
since $V_{11} \in \mathcal I_0$ and $U_{12} \in \mathcal A$, while 
$(UV)_{2,2} = U_{21}V_{12} + U_{22}V_{22} \in \mathcal A(Y)$ 
since $V_{12}, V_{22} \in \mathcal A$.
The strict inclusions (i) - (iii) for the component ideals imply that all 
the ideals in the above list are non-trivial, as well as pairwise different. In particular, 
$\mathcal A(X \oplus Y)  \varsubsetneq \mathcal I \cap \mathcal J$ by (iii). 

\smallskip

A closer inspection reveals the order structure among the ideals in the above list is the following, 
where a line indicates a strict inclusion (from left to right):

\medskip

\begin{center}
\begin{tikzcd}[tips=false,column sep=0.6em,row sep=0.6em]
& &\mathcal I_2 \ar{dl}\ar{r}  \ar{dr} &\mathcal I\ar{r}&\mathcal I_1\ar{dl}& \\
\mathcal A(X\oplus Y)\ar{r}&\mathcal I\cap \mathcal J\ar{dr}&&\mathcal I_1\cap \mathcal J_1\ar{dr} \ar{dl}&&\mathcal K(X\oplus Y).\ar{dl}\ar{ul}\\
&&\mathcal J_2\ar{r}& \mathcal J\ar{r}&\mathcal J_1&  \\
\end{tikzcd}
\end{center}

\medskip

\noindent Thus  $\mathcal K(X \oplus Y)$ contains several chains of closed ideals,
as well as incomparable pairs. $\Box$
\end{ex}

Let $1 \le p < q < \infty$. Recall that the closed ideals 
\[
 \left( \begin{array}{ccc}
\mathcal L(\ell^p) & \mathcal K(\ell^q,\ell^p) \\
\mathcal L(\ell^p,\ell^q)  &  \mathcal K(\ell^q) \\
\end{array} \right)   \textrm{ and } 
 \left( \begin{array}{ccc}
\mathcal K(\ell^p) & \mathcal K(\ell^q,\ell^p) \\
\mathcal L(\ell^p,\ell^q)  &  \mathcal L(\ell^q) \\
\end{array} \right)  
\]
are maximal in $\mathcal L(\ell^p \oplus \ell^q)$, see e.g.
\cite[5.3.2]{Pie}. Suppose that  $X \subset \ell^p$ and $Y \subset \ell^q$ are closed subspaces,
where $p < 2 < q$.
It is a natural question whether the corresponding closed ideals 
\[
\mathcal I =  \left( \begin{array}{ccc}
\mathcal K(X) & \mathcal A(Y,X) \\
\mathcal K(X,Y)  &  \mathcal A(Y) \\
\end{array} \right)   \textrm{ and } 
\mathcal J =  \left( \begin{array}{ccc}
\mathcal A(X) & \mathcal A(Y,X) \\
\mathcal K(X,Y)  &  \mathcal K(Y) \\
\end{array} \right) 
\]
given by Theorem \ref{ideals} are maximal in $\mathcal K(X \oplus Y)$. 
The following variant of the preceding examples 
demonstrate that neither $\mathcal I$ nor $\mathcal J$ are in general 
maximal ideals in the setting of Theorem \ref{ideals}. 

\begin{ex}\label{nonmax}
Suppose that $1 \le p < 2$ and $q > 4$. Let $W \subset \ell^q$ be the 
closed subspace from Example \ref{newideal} 
such that \eqref{ideal} and \eqref{dideal} hold.
In view of Facts \ref{ap1}.(ii) let  $X \subset \ell^p$ be
a closed subspace  such that 
$\mathcal A(X) \subsetneq \mathcal K(X)$. The direct sum 
$X \oplus W$ 
satisfies the assumptions of Theorem \ref{ideals}, so that 
$\mathcal K(W,X) = \mathcal A(W,X)$. 
Let $\mathcal J_0 = \overline{\mathcal{QN}_q(W)}$. 
By arguing as in Example \ref{8ideal}  we get that 
\[
\mathcal I =  \left( \begin{array}{ccc}
\mathcal K(X) & \mathcal A(W,X) \\
\mathcal K(X,W)  &  \mathcal A(W) \\
\end{array} \right) \varsubsetneq 
\left( \begin{array}{ccc}
\mathcal K(X) & \mathcal A(W,X) \\
\mathcal K(X,W)  & \mathcal J_0 \\
\end{array} \right)
\]
are both non-trivial closed ideals in $\mathcal K(X \oplus W)$. 
Thus $\mathcal I$ is not a maximal 
ideal in $\mathcal K(X \oplus W)$, where $X \oplus W$
is a closed subspace of $\ell^p \oplus  \ell^q$.

Actually, it is not difficult to check that the Banach algebra
$\mathcal K(X \oplus W)$ contains (at least) the following 
four non-trivial  closed ideals:
$\mathcal I$ and $\mathcal J$ from Theorem \ref{ideals}, together with
\[
\left( \begin{array}{ccc}
\mathcal  A(X) & \mathcal A(W,X) \\
\mathcal K(X,W)  &  \mathcal J_0\\
\end{array} \right)  \textrm{ and } 
\left( \begin{array}{ccc}
\mathcal  K(X) & \mathcal A(W,X) \\
\mathcal K(X,W)  & \mathcal J_0\\
\end{array} \right).
\]

Next, let $r > 2$ and pick a closed subspace $Y \subset \ell^r$ such that 
$\mathcal A(Y) \subsetneq \mathcal K(Y)$. Then $W^* \oplus Y$
satisfies the assumptions of Theorem \ref{ideals}. As in Example 
\ref{8ideal} we get with $\mathcal I_0 = \overline{\mathcal{SK}_p(W^*)}$
from \eqref{dideal} that 
\[
\mathcal J =  \left( \begin{array}{ccc}
\mathcal A(W^*) & \mathcal A(Y,W^*) \\
\mathcal K(W^*,Y)  &  \mathcal K(Y) \\
\end{array} \right) \varsubsetneq 
\left( \begin{array}{ccc}
\mathcal I_0 & \mathcal A(Y,W^*) \\
\mathcal K(W^*,Y)  &  \mathcal K(Y)\\
\end{array} \right)
\]
are again non-trivial closed ideals of $\mathcal K(W^* \oplus Y)$. 
In particular, $\mathcal J$ is not maximal  in 
$\mathcal K(W^* \oplus Y)$. In this case $W^*$ is a quotient space 
of $\ell^{q'}$.  $\Box$
\end{ex}

\smallskip

Recently Kim \cite{Kim19} studied the $\mathcal K_{p,q}^{inj}$-A.P.  and the 
uniform $\mathcal K_{p,q}^{inj}$-A.P.
for a scale $\mathcal K_{p,q}$ of Banach operator ideals that include the classical 
$p$-compact operators as $\mathcal K_p =: \mathcal K_{p,p'}$,
where $p'$ is the dual exponent of $p$. In particular,
he obtained the surprising result \cite[Proposition 4.3]{Kim19} that 
$X$ has the $\mathcal K_p^{inj}$-A.P. if and only if $X$ has the 
uniform $\mathcal K_p^{inj}$-A.P. However, \cite{Kim19} 
does not include explicit examples concerning
this approximation property. Recall that $\mathcal{KS}_p = \mathcal K_p^{inj}$
by Proposition \ref{Kfactor1}, so that the following examples about
the  (uniform) $\mathcal K_p^{inj}$-A.P. are contained in
Proposition \ref{Bfactor}, Theorem \ref{Figiel}  and Theorem \ref{reinov}.

\begin{ex}\label{kim}
Let $1 \le p < \infty$ and $p \neq 2$.

(i) If $X \subset \ell^p$ is a closed subspace, then $X$ has the (uniform) 
$\mathcal{KS}_p$-A.P. if and only if $X$ has the A.P.

(ii) Let $2 < q < \infty$. If $X$ has cotype $2$, then $X$ has the 
(uniform) $\mathcal{KS}_q$-A.P. In particular, for $1 \le p < 2$ 
there is a closed subspace $X\subset \ell^p$
that has the (uniform) $\mathcal{KS}_q$-A.P., but fails the A.P.

(iii) If $p,q, \in [1,2)$ or $p, q \in (2,\infty)$, then there is a closed 
subspace $X \subset \ell^p$ such that $X$ fails the (uniform) $\mathcal{KS}_q$-A.P.
\end{ex}

\begin{proof}
(i) If  $X \subset \ell^p$ is a closed subspace, then 
$\mathcal K(Z,X) = \mathcal {KS}_p(Z,X)$ for any space $Z$ 
by Proposition \ref{Bfactor}.(i).
Hence, if $X$ has the (uniform) $\mathcal{KS}_p$-A.P., then 
\[
 \mathcal K(Z,X) =  \mathcal {KS}_p(Z,X) = \mathcal A(Z,X)
\]
for all  $Z$, so that $X$ has the A.P. Conversely, if $X$ has the A.P., then
\[
\mathcal{KS}_p(Z,X) \subset \mathcal K(Z,X) =  \mathcal A(Z,X)
\]
for any $Z$, so that $X$ has the uniform $\mathcal{KS}_p$-A.P.
(The converse implication is also noted in \cite[Corollary 4.5]{Kim19}.)

\smallskip

(ii) Let  $Z$ be an arbitrary Banach space, and suppose that $T \in \mathcal{KS}_q(Z,X)$.
Hence there is a closed subspace $M \subset \ell^q$ and a compact factorisation $T = BA$, 
where $B \in \mathcal K(M,X)$.  Theorem \ref{KMJ} implies that  $B \in \mathcal A(M,X)$,
since $X$ has cotype $2$ and $M$ has type $2$. 
This means that $\mathcal {KS}_q(Z,X) \subset \mathcal A(Z,X)$, so  that 
$X$ has the (uniform) $\mathcal{KS}_q$-A.P. Finally, recall from 
 Facts \ref{ap1}.(i) that  for $1 \le p < 2$ there are closed subspaces 
$X \subset \ell^p$ that fail the A.P. 

\smallskip

(iii)  We may assume that $p \neq q$, since the case $p = q$ follows from part (i) together with 
Facts \ref{ap1}.(i).
If $1 \le q < p < 2$ , then recall from Theorem \ref{Figiel}.(i) that 
$\mathcal{KS}_q(X) = \mathcal K(X)$ for any closed subspace $X \subset \ell^p$. 
Hence $X$ fails the uniform $\mathcal{KS}_q$-A.P. whenever 
$X \subset \ell^p$ is a closed subspace such that $\mathcal A(X) \varsubsetneq \mathcal K(X)$,
where such subspaces again exists by Facts \ref{ap1}.(ii).

Let $1 \le p < q < 2$. According to the proof of Theorem \ref{Figiel}.(iii) 
there is a closed subspace $X \subset \ell^p$ together with  an operator 
\[
 U = AB \in \mathcal{KS}_q(X) \setminus  \mathcal A(X).
\]
This means that $X$ does not have the uniform $\mathcal{KS}_q$-A.P.

Let $p, q \in (2,\infty)$. By Theorem \ref{reinov} 
there is a closed subspace  $X\subset \ell^p$ 
such that $X$ fails the uniform $\mathcal{QN}_q$-A.P.
Consequently $X$ also fails the (uniform) $\mathcal{KS}_q$-A.P.
in view of  \eqref{qnfactor}.
\end{proof}
 
\section{Concluding examples and problems}\label{quest}

\smallskip

In this section we find a natural Banach operator ideal, which is equipped with 
the operator norm, that induces a non-trivial 
closed ideal of $\mathfrak{A}_X$ for a class of Banach spaces $X$. Moreover,
we exhibit a class of spaces $X$ such that $\mathfrak{A}_X$ has uncountably many 
closed ideals with an explicit order structure. Both examples are related to the failure of 
duality for the A.P.

Let $X$ and $Y$ be Banach spaces. The operator $T \in \mathcal L(X,Y)$
is \textit{compactly approximable}, denoted $T \in \mathcal{CA}(X,Y)$,
if for any compact subset $K \subset X$ and $\varepsilon > 0$ there is a bounded finite rank 
operator $V \in \mathcal F(X,Y)$ so that 
\[
\sup_{x \in K} \Vert Tx - Vx\Vert < \varepsilon.
\]
In other words, $\mathcal{CA}(X,Y) = \overline{\mathcal F(X,Y)}^{\tau}$, 
where the closure is taken in $\mathcal L(X,Y)$ with respect to the 
topology $\tau$  of uniform convergence on compact sets in $\mathcal L(X,Y)$.
By definition  $X$ has the A.P. if and only if the identity operator $I_ X \in \mathcal{CA}(X)$. 

The class $\mathcal{CA}$ was used by Pisier \cite{P80} 
(see also \cite[0.2]{P2} and \cite[31.5]{DF}). 
It defines a Banach operator ideal,  which has not been much studied, though 
the related class $\overline{\mathcal K(X,Y)}^{\tau}$ 
appears in \cite{GS} and \cite{Go}. 
We first list the relevant basic properties of $\mathcal{CA}$.

\begin{prop}\label{ca}
(i) $\mathcal{CA}$ is a closed Banach operator ideal.

(ii) If $X$ or $Y$ has the A.P., then
 $\mathcal{CA}(X,Y) = \mathcal L(X,Y)$.
 
 (iii) $\mathcal{CA}(X,Y) = \overline{\mathcal A(X,Y)}^{\tau}$.

(iv) If $V \in \mathcal K(Z,X)$ and $U \in \mathcal{CA}(X,Y)$, then $UV \in \mathcal A(Z,Y)$.
\end{prop}

\begin{proof}
(i) It is straightforward to check the operator ideal properties of $\mathcal{CA}$,
and we leave this to the reader.  Suppose that $U \in \overline{\mathcal{CA}(X,Y)}$. 
Let  $K \subset X$ be a 
compact subset, $\varepsilon  > 0$ and put $M = \sup_{x \in K} \Vert x\Vert$.
First pick $T \in \mathcal{CA}(X,Y)$ such that $\Vert U - T \Vert < \varepsilon /(M+1)$, and then
$V \in \mathcal F(X,Y)$ such that
 $\Vert Tx - Vx \Vert < \varepsilon$ for all $x \in K$. Hence, for any $x \in K$ we have 
\[
\Vert Ux - Vx\Vert \le M \Vert U- T\Vert  + \Vert Tx - Vx\Vert < 2 \varepsilon.
\]
It follows that   $U \in \mathcal{CA}(X,Y)$. 

(ii) By assumption $I_X$ or $I_Y$ is compactly approximable, 
so that $\mathcal{CA}(X,Y) = \mathcal L(X,Y)$ by the
operator ideal property of  $\mathcal{CA}$.

(iii) This is an simple variant of the argument in part (i). 

(iv) Let $\varepsilon > 0$ be arbitrary. 
Since $\overline{VB_Z}$ is a compact subset of $X$,
there is $T \in \mathcal F(X,Y)$ so that
\[
\Vert UV - TV\Vert  = \sup_{z \in \overline{VB_Z}} \Vert Uz- Tz\Vert < \varepsilon.
\]
We obtain that $UV \in \mathcal A(Z,Y)$
since $TV \in \mathcal F(Z,Y)$.
\end{proof}

We next use the failure of duality for the approximation property to show
that the closed Banach operator ideal $\mathcal{CA} \cap  \mathcal K$
gives a non-trivial closed ideal inside the compact operators for certain Banach spaces. 
Recall that there are Banach spaces $X$ such that 
$X$ has the A.P., but $X^*$ fails to have the A.P. In fact, 
for every separable Banach space $Y$ that fails the A.P. 
there is by \cite[Theorems 1.d.3 and  1.e.7.(b)]{LT1} a Banach space $Z$ such that 
$Z^{**}$ has a Schauder basis and $Z^{**}/Z$ is isomorphic to $Y$, so that 
$Z^{***} \approx Z^* \oplus Y^*$ fails the A.P.  (since $Y^*$ also fails the A.P.)
There are also concrete spaces of this kind:  the space  $X = \mathcal N_1(\ell^2)$
of the $1$-nuclear operators on $\ell^2$ has a Schauder basis, 
but $X^* = \mathcal L(\ell^2)$ fails the A.P. by a celebrated result of Szankowski \cite{Sz81}.
Recall further that spaces $X$ having the above property
cannot be reflexive, cf. \cite[Theorem 1.e.7.(a)]{LT1}.

\begin{ex}\label{cak}
Let $X$ be a Banach space such that $X$ has the A.P., 
but $X^*$ fails the A.P.  By \cite[Theorem 1.e.5]{LT1}
there is a Banach space $Y$ such that $\mathcal A(X,Y) \varsubsetneq \mathcal K(X,Y)$. 
Suppose that  $W = X \oplus Y \oplus Z$, where either

\smallskip

\begin{itemize}
\item [(i)]
the Banach space $Z$ has the B.C.A.P., but fails to have the 
A.P. (such spaces were first constructed by Willis \cite{W}), or

\item [(ii)]
$Z \subset \ell^p$ is the closed subspace constructed in  Theorem \ref{nil}
for $1 \le p < \infty$ and $p \neq 2$.
\end{itemize}

\smallskip

\noindent \textit{Claim.} 
For $\mathcal I =:  \mathcal{CA} \cap  \mathcal K$ we have 
\[
\mathcal A(W)  \varsubsetneq    \mathcal I(W)   \varsubsetneq  \mathcal K(W),
\]
where the induced quotient ideal $q(\mathcal I(W))$ is nilpotent in $\mathfrak{A}_W$. 
Moreover, there is $V \in \mathcal K(W)$ such that
$V^n \notin \mathcal I(W)$ for any $n \in \mathbb N$, so that the radical quotient algebra
$\mathcal K(W)/\mathcal I(W)$ is non-nilpotent and infinite-dimensional. 
\end{ex}

\begin{proof}
We will again systematically use the fact that Banach operator ideals are uniquely determined 
on finite direct sums  by their respective ideal components. 
Since $X$ has the A.P., it follows from Proposition \ref{ca}.(ii) that 
$\mathcal{CA}(X,Y) = \mathcal L(X,Y)$. Hence 
\[
\mathcal A(X,Y)  \varsubsetneq  \mathcal{CA}(X,Y) \cap \mathcal{K}(X,Y) = \mathcal{K}(X,Y),
\]
as  $\mathcal A(X,Y) \varsubsetneq \mathcal K(X,Y)$ by our choice of $X$ and $Y$. 
This implies that  $\mathcal A(W)  \varsubsetneq \mathcal I(W)$.

We need the fact that there is a compact operator 
$U \in \mathcal K(Z)$ such that $U^n \notin \mathcal A(Z)$ for all $n \in \mathbb N$. 
This follows from \cite[Proposition 3.1]{TW} in the case (i) and from
Theorem \ref{nil} in the case (ii).
Hence $U \notin \mathcal{CA}(Z)$, since otherwise
$U^2 \in \mathcal A(Z)$ by Proposition  \ref{ca}.(iv). It follows that 
\[
U \in \mathcal K(Z) \setminus (\mathcal{CA}(Z) \cap \mathcal{K}(Z)),
\] 
so that also $\mathcal I(W)  \varsubsetneq \mathcal K(W)$.

Define $V \in \mathcal K(W)$ by $V(x,y,z) = (0,0,Uz)$ for $(x,y,z) \in W$. 
The above properties yield that $V^n \notin \mathcal{CA}(W)$ for $n \in \mathbb N$,
so that the quotient algebra $\mathcal K(W)/\mathcal I(W)$ is non-nilpotent.
By  Proposition \ref{real} the quotient $\mathcal K(W)/\mathcal I(W)$ is a radical algebra 
both in the real and complex cases. 
It follows once more from  \cite[Proposition 1.5.6.(iv)]{D00} that 
the algebra $\mathcal K(W)/\mathcal I(W)$ is infinite-dimensional.

Finally,   Proposition  \ref{ca}.(iv) implies  that  $ST \in \mathcal A(W)$ for any
$S, T \in \mathcal{CA}(W) \cap \mathcal{K}(W)$, so that
$q(\mathcal I(W))$ is a nilpotent closed ideal of $\mathfrak{A}_W$.
\end{proof}

The preceding example also demonstrates
that in Proposition  \ref{ca}.(iv) the order of composition matters.
Namely, if we pick  
$U \in  \mathcal K(X,Y) \setminus  \mathcal A(X,Y)$ according to Example
\ref{cak}, then $I_X \in \mathcal{CA}(X)$ 
as $X$ has the A.P., but 
$U \circ I_X \notin  \mathcal A(X,Y)$.

\smallskip

In Example \ref{cak} the space $W = X \oplus Y \oplus Z$ is not reflexive, 
since $X$ cannot be reflexive  as we noted above.
We observe next that actually $\mathcal{CA} \cap \mathcal K$ coincides with 
$\mathcal A$ within the class of reflexive Banach spaces.
This is based on a representation by Godefroy and Saphar \cite{GS}  
of the bidual $\mathcal K(X,Y)^{**}$. 

\begin{prop}\label{caideal}
If $X$ is a reflexive Banach space, then 
\[
\mathcal{CA}(X) \cap \mathcal K(X) = \mathcal A(X).
\]
\end{prop}

\begin{proof}
Let $X$ be any reflexive Banach space. We require the fact that, 
up to isometric isomorphism, we may identify
\begin{equation}\label{cafact}
\mathcal A(X)^{**} = \overline{\mathcal A(X)}^{\tau} = \mathcal{CA}(X),
\end{equation}
where the closure is taken in $\mathcal L(X)$ with respect to the topology $\tau$ 
of uniform convergence on compact subsets $K \subset X$.
This fact is essentially contained in the proofs of 
Proposition 1.1, Corollaries 1.2 and 1.3 in \cite{GS}.

Namely, by \cite[Proposition 1.1]{GS} we have
$\mathcal K(X)^{**} =  \overline{\mathcal K(X)}^{w^*}$ up to isometric isomorphism,
where the $w^*$-topology denotes the one induced in $\mathcal L(X)$ from the
duality $(X \hat{\otimes}_{\pi} X^*)^* = \mathcal L(X)$. In fact,
let $(T_\alpha) \subset \mathcal K(X)^{**}$ be an arbitrary net and $T \in \mathcal K(X)^{**}$. 
According to \cite[Proposition 1.1]{GS} there is for reflexive spaces $X$ a quotient map 
$Q: X \hat{\otimes}_{\pi} X^* \to \mathcal K(X)^*$ such that $Q^*$ defines an isometric
embedding $\mathcal K(X)^{**} \to \mathcal L(X)$.
In particular, 
\[
\langle z, Q^*T_\alpha - Q^*T \rangle = \langle Qz,T_\alpha - T \rangle \textrm{ for all } 
z \in X \hat{\otimes}_{\pi} X^*.
\]
Moreover, by \cite[Corollary 1.3]{GS} 
we may identify
$\mathcal K(X)^{**} =  \overline{\mathcal K(X)}^{\tau}$ in $\mathcal L(X)$. 
By standard duality we also have $\mathcal A(X)^{**} = \mathcal A(X)^{\perp \perp}
\subset \mathcal K(X)^{**}$,
where $\mathcal A(X)^{\perp \perp}$ denotes the biannihilator of $\mathcal A(X)$
in $\mathcal K(X)^{**}$. 
Since the $w^*$-topology of $\mathcal L(X)$ and the $\tau$-topology 
have the same closed convex sets in $\mathcal L(X)$, see
the proof of \cite[Corollary 1.2]{GS}, we deduce that \eqref{cafact} holds.

Recall next that if $Z$ is a Banach space and $M \subset Z$ is a closed subspace, 
then a lemma of Valdivia \cite[pp. 107-108]{V87} (see also 
\cite[Lemma A.5.1]{GMA}) implies that
$\overline{M}^{w^{*}} \cap Z = M$, where $w^*$ denotes the $w^*$-topology of $Z^{**}$.
By applying these facts to $\mathcal A(X) \subset \mathcal K(X)$ we obtain that 
\[
\mathcal{CA}(X) \cap \mathcal K(X) = \mathcal A(X)^{**} \cap \mathcal K(X) = \mathcal A(X)
\]
which completes the argument.
\end{proof} 

\begin{remarks}
(i) Proposition \ref{caideal} does not extend to the class of Banach spaces 
that has  the Radon-Nikod\'ym property (RNP), since there are  direct sums 
$W = X \oplus Y \oplus Z$ with the RNP among those of Example \ref{cak}. 
We refer to e.g. \cite[Chapter III]{DU} for the definition of the RNP and 
the facts that separable dual spaces  as well as reflexive spaces have the RNP.

Namely,  $X$ can be chosen a separable dual space  in Example \ref{cak} 
and $Y$ a reflexive space by applying the DFJP-factorisation theorem 
(see e.g. \cite[Theorem 2.g.11]{LT2} or \cite[Theorem 3.2.1]{GMA}). 
In condition (ii) of Example  \ref{cak} the subspace $Z \subset \ell^p$ is reflexive
for $p > 1$, while there are also separable 
reflexive spaces $Z$ which satisfy condition (i) by \cite[Propositions 3 and 4]{W}. 

\smallskip

(ii) In the direct sum $W = X \oplus Y \oplus Z$ of Example \ref{cak} 
it is also possible to choose $Z$ reflexive so
that $\mathfrak{A}_Z \neq \{0\}$. Then Proposition \ref{caideal} implies 
that $\mathcal{CA}(Z)\cap \mathcal K(Z)= \mathcal A(Z) \varsubsetneq \mathcal K(Z)$,
and hence
$\mathcal I(W) = \mathcal{CA}(W)\cap \mathcal K(W) \varsubsetneq \mathcal K(W)$.
However, such a choice does not by itself guarantee that
$\mathcal K(W)/\mathcal I(W)$ is a non-nilpotent quotient algebra,  
which is also obtained in  Example \ref{cak}.
\end{remarks}

\medskip

We finally exhibit a quite dramatic example, 
where the compact-by-approximable algebra  contains an uncountable family of closed ideals having the reverse lattice structure of the partially ordered 
power set $(\mathcal P(\mathbb N),\subset)$ of the natural numbers $\mathbb N$.
This example has a very special form  and in a sense it is the most elementary one
contained here.
 
\begin{ex}\label{infty}
Fix  $1 < p < \infty$ and let $X$ be any Banach space 
such that $X$ has the A.P., but $X^*$ fails to have the 
A.P. By \cite[Theorem 1.e.5]{LT1} there is a Banach space $Y_0$ such that 
\[
\mathcal A(X,Y_0) \varsubsetneq \mathcal K(X,Y_0),
\] 
but where $\mathcal K(Y_0,X) = \mathcal A(Y_0,X)$ and $\mathcal K(X) = \mathcal A(X)$, 
since $X$ has the A.P.  By replacing $Y_0$ with the
direct $\ell^p$-sum $Y =: \big(\oplus_{\mathbb N} Y_0)_p$, 
we may assume that  the quotient space $\mathcal K(X,Y)/\mathcal A(X,Y)$ 
is infinite-dimensional. In fact, let $S \in \mathcal K(X,Y_0)$ satisfy
$\Vert S\Vert  = 1$ and $dist(S,\mathcal A(X,Y_0)) > 1/2$. Define 
$S_j: X \to Y$ by 
\[
S_jx = (0,\ldots,0,Sx,0,\ldots), \quad x \in X \textrm{ and } j \in \mathbb N,
\] 
where the vector $Sx$ belongs to the $j$:th component of $Y$. 
It is straightforward to check that 
$dist(S_i-S_j,\mathcal A(X,Y)) \ge 1/2$ for $i \neq j$, so that 
$\mathcal K(X,Y)/\mathcal A(X,Y)$  is infinite-dimensional. Note that we still have
$\mathcal K(Y,X) = \mathcal A(Y,X)$. 

We consider the direct sum 
\[
Z = \big(\oplus_{k=0}^\infty  X_k\big)_p  \ ,
\]
where we put $X_0 = Y$ and 
$X_k = X$ for $k \in \mathbb N$ for unicity of notation.
Let $P_k \in \mathcal L(Z)$ be the natural  projection map onto $X_k$ and let $J_k: X_k \to Z$ 
be the corresponding natural inclusion map
for $k \in \mathbb N_0 =: \mathbb N \cup \{0\}$.
We extend our operator matrix notation to operators on $Z$, and write 
$U = (U_{r,s}) \in  \mathcal K(Z)$, where 
\[
U_{r,s} = (U)_{r,s} =: P_rUJ_s \in \mathcal K(X_s,X_r) \textrm{ for }  r, s \in \mathbb N_0.
\]
(It should be noticed that if we are given
$U_{r,s} \in \mathcal K(X_s,X_r)$ for all $r,s \in \mathbb N_0$,
then it is a separate question whether the formal operator matrix 
$U = (U_{r,s})$ defines a compact or even a bounded operator on $Z$.) 

For any given 
subset $A \subset \mathbb N$ we define
\[
\mathcal I_A = \{U = (U_{r,s}) \in  \mathcal K(Z):  U_{0,0} \in  \mathcal A(Y) \textrm{ and } 
U_{0,s} \in  \mathcal A(X,Y) \textrm{ for all } s \in A\}.
\]
Observe that if  $U = (U_{r,s}) \in  \mathcal K(Z)$, 
then $U_{r,s} \in \mathcal A(X_s,X_r)$ whenever $r \ge 1$, since 
$\mathcal K(Y,X) = \mathcal A(Y,X)$ and $\mathcal K(X) = \mathcal A(X)$ by construction. 
Hence only the components $U_{0,s} \in \mathcal K(X,Y)$  
for $s \in \mathbb N$ will play an explicit role  in the definition 
of $\mathcal I_A$.

\smallskip

\textit{Claim}:  the family $\{\mathcal I_A: A \subset \mathbb N\}$ has the following 
properties:

\smallskip

\begin{itemize}
\item [(i)] $\mathcal I_A$ is a closed ideal of $\mathcal K(Z)$ satisfying
$\mathcal A(Z) \varsubsetneq   \mathcal I_A  \varsubsetneq  \mathcal K(Z)$ 
for all subsets $\emptyset \neq A \varsubsetneq \mathbb N$. In addition,
$\mathcal I_{\mathbb N} = \mathcal A(Z)$ and right multiplication on $\mathcal I_A$ 
satisfies $US \in \mathcal A(Z)$ for
$U \in \mathcal I_A$ and $S \in \mathcal K(Z)$.  In particular, the quotient ideals
$q( \mathcal I_A)$ are nilpotent in $\mathfrak{A}_Z$ whenever 
$\emptyset \neq A \varsubsetneq \mathbb N$.

\item  [(ii)]  If $A \subset B$, then $\mathcal I_B  \subset  \mathcal I_A$. Moreover, if 
$A \varsubsetneq B$, then $\mathcal I_A / \mathcal I_B$ is infinite-dimensional.

\item  [(iii)]  $\mathcal I_A \cap \mathcal I_B = \mathcal I_{A \cup B}$ and 
$\mathcal I_A + \mathcal I_B = \mathcal I_{A \cap B}$ for all subsets 
$A, B \subset \mathbb N$. In particular, 
$\mathcal I_A + \mathcal I_B$ is closed in $\mathcal K(Z)$ for all 
$A, B \subset \mathbb N$.
\end{itemize}

\smallskip

Hence  the family 
$\{\mathcal I_A: \emptyset \neq A \varsubsetneq \mathbb N\}$ 
of non-trivial closed ideals of $\mathcal K(Z)$
does not have a smallest or 
a largest element, and $\mathcal I_{\{k\}}$ are maximal elements of this family
for each $k \in \mathbb N$.
By Proposition \ref{quotient} the family $\{q(\mathcal I_A): 
\emptyset \neq A \varsubsetneq \mathbb N\}$ is an uncountable family of 
non-trivial closed ideals of $\mathfrak{A}_Z$ having the reverse partial order structure of 
$(\mathcal P(\mathbb N),\subset)$.

\end{ex}

\begin{proof}
(i) For $s \in \mathbb N_0$ define the bounded linear map 
$\psi_{s}: \mathcal K(Z) \to \mathcal K(X_s,Y)$ by $\psi_{s}(S) = P_0SJ_s$
for $S \in \mathcal K(Z)$. Deduce that
\[
\mathcal I_A = \bigcap_{s \in A \cup \{0\}} \psi_{s}^{-1}(\mathcal A(X_s,Y))
\]
is a closed linear subspace of $\mathcal K(Z)$. 

Let $U  = (U_{r,s})  \in \mathcal I_A$ and 
$S = (S_{r,s}) \in \mathcal K(Z)$ be arbitrary. 
We next claim that $SU \in \mathcal I_A$ and $US \in  \mathcal I_A$, which means  that 
$\mathcal I_A$ is a closed ideal of $\mathcal K(Z)$. Note first that 
$SU \in \mathcal K(Z)$ so that the components 
$(SU)_{r,s} \in \mathcal K(X_s,X_r)$ for all $r,s \in \mathbb N_0$, and
$(SU)_{r,s} \in \mathcal K(X_s,X_r) = \mathcal A(X_s,X_r)$ for $r \in \mathbb N$. 
Similar facts hold for the components  $(US)_{r,s}$.

In order to verify that $(SU)_{0,s} \in \mathcal A(X_s,Y)$ for $s \in A \cup \{0\}$, observe that 
the finite sums 
\[
\sum_{k=0}^N S_{0,k}U_{k,s} = \sum_{k=0}^N P_0SJ_kP_kUJ_s \in \mathcal A(X_s,Y)
\]
for all $N \in \mathbb N$,  since 
$U_{0,s} = P_0UJ_s \in \mathcal A(X_s,Y)$ for $s \in A$ and $U_{0,0} \in \mathcal A(Y)$
by assumption. Moreover, 
\[
\Vert P_0SUJ_s - \sum_{k=0}^N P_0SJ_kP_kUJ_s\Vert 
= \Vert P_0S\big(U -  \sum_{k=0}^N J_kP_kU\big)J_s\Vert
\le \Vert S\Vert \cdot \Vert U -  \sum_{k=0}^N J_kP_kU\Vert \to 0
\]
as $N \to \infty$. Here $\Vert U -  \sum_{k=0}^N J_kP_kU\Vert \to 0$ as $N \to \infty$,
because the sequence $\sum_{k=0}^N J_kP_k \to I_Z$ pointwise on $Z$
as $N \to \infty$ and $U \in \mathcal K(Z)$ is compact. 
In a similar manner one deduces that $(US)_{0,s} \in \mathcal A(X_s,Y)$
for all $s \in \mathbb N_0$. In this case 
 $\sum_{k=0}^N U_{0,k}S_{k,s} \in \mathcal A(X_s,Y)$ for all  
$s \in \mathbb N_0$ and all $N$, since we also have $U_{0,0} \in \mathcal A(Y)$.

If $U = (U_{r,s}) \in \mathcal I_{\mathbb N}$ is arbitrary, 
then $U_{r,s} \in \mathcal A(X_s,X_r)$
for all $r, s \in \mathbb N_0$. It follows that $U \in \mathcal A(Z)$ from
the general fact proved separately below  in Lemma \ref{sum}, so that 
$\mathcal I_{\mathbb N} = \mathcal A(Z)$.  Thus the above argument 
says  that 
$US \in \mathcal I_{\mathbb N} = \mathcal A(Z)$ whenever $U \in \mathcal I_A$, 
$S \in \mathcal K(Z)$, and  $\emptyset \neq A \varsubsetneq \mathbb N$.
In particular, $UV \in \mathcal A(Z)$ whenever $U, V \in \mathcal I_A$, so that
$q( \mathcal I_A)$ defines a nilpotent ideal in $\mathfrak{A}_Z$.

Finally we verify that $\mathcal A(Z) \varsubsetneq   \mathcal I_A  \varsubsetneq  
\mathcal K(Z)$  for  
$\emptyset \neq A \varsubsetneq \mathbb N$. First fix $s \in A$ and pick 
$U_{0,s} \in \mathcal K(X_s,Y) \setminus \mathcal A(X_s,Y)$. Define 
$U = (U_{r,t}) \in \mathcal K(Z)$ by $U_{r,t} = 0$ for $(r,t) \neq (0,s)$,
so that $U \notin \mathcal I_A$. Next  fix $t \in A^c$ and pick 
$V_{0,t} \in \mathcal K(X_t,Y) \setminus \mathcal A(X_t,Y)$. 
If we define $V = (V_{r,s}) \in \mathcal K(Z)$ by $V_{r,s} = 0$ for $(r,s) \neq (0,t)$,
then $V \in \mathcal I_A \setminus \mathcal A(Z)$. 

\medskip

\noindent (ii) Clearly  $\mathcal I_B  \subset  \mathcal I_A$ if $A \subset B$. 
Let $k \in B \setminus A$.
Since the quotient $\mathcal K(X_k,Y)/\mathcal A(X_k,Y)$ is infinite-dimensional by construction, 
there is  a normalised sequence $(S^{(j)}) \subset \mathcal K(X_k,Y)$ such that 
\[
dist(S^{(j)} - S^{(i)},\mathcal A(X_k,Y)) > 1/2 \textrm{ for all } i \neq j.
\] 
Define a normalised sequence $(V_j) \subset \mathcal K(Z)$ 
such that the components of $V_j = (V_{r,s}^{(j)})$ satisfy
$V_{0,k}^{(j)} = S^{(j)}$ and $V_{r,s}^{(j)} = 0$ for $(r,s) \neq (0,k)$.
Then $V_j \in \mathcal I_A \setminus \mathcal I_B$ for all $j \in \mathbb N$.
Let $S = (S_{r,s}) \in \mathcal I_B$ be arbitrary. It follows that 
\[
\Vert V_j - V_i - S\Vert \ge \Vert P_0(V_j - V_i - S)J_k\Vert
= \Vert S^{(j)} - S^{(i)} - S_{0,k}\Vert > 1/2
\]
for all $i \neq j$, since $S_{0,k} \in \mathcal A(X_k,Y)$. 
Conclude that $dist(V_j - V_i,\mathcal I_B) \ge 1/2$ for $i \neq j$, so that
the quotient  $\mathcal I_A / \mathcal I_B$ is infinite-dimensional.

\medskip

\noindent (iii) Let $A, B \subset \mathbb N$ be non-empty subsets. If 
$U =  (U_{r,s}) \in \mathcal I_A \cap \mathcal I_B$, then $U_{0,k} \in \mathcal A(X_k,Y)$ 
for all $k \in A \cup B$, so that $ \mathcal I_A \cap \mathcal I_B \subset \mathcal I_{A \cup B}$.
Conversely,  $\mathcal I_{A \cup B} \subset \mathcal I_A$ 
and $\mathcal I_{A \cup B} \subset \mathcal I_B$  by monotonicity.

We also get that $\mathcal I_A \subset \mathcal I_{A \cap B}$ and 
$\mathcal I_B \subset \mathcal I_{A \cap B}$ by monotonicity, 
so that $\overline{\mathcal I_A + \mathcal I_B} \subset \mathcal I_{A \cap B}$, 
because $\mathcal I_{A \cap B}$ is a closed ideal. 
Conversely, let $U = (U_{r,s}) \in \mathcal I_{A \cap B}$ be arbitrary.
We define the bounded operator $\theta_A \in \mathcal L(Z)$  by 
\[
\theta_Az = (\chi_A(k)z_k)  \textrm{ for } z = (z_k) \in Z,
\]
where $\chi_A$ is the characteristic function of $A \subset \mathbb N_0$.
Thus  $U\theta_A \in \mathcal K(Z)$ and
\[
U =  \big(U - U\theta_A\big) + U\theta_A 
\]
so it will be enough to verify that $U\theta_A  \in \mathcal I_B$ and 
$U - U\theta_A \in \mathcal I_A$. In fact,
this immediately yields that $\mathcal I_A + \mathcal I_B = \mathcal I_{A \cap B}$, 
and hence that $\mathcal I_A + \mathcal I_B$ 
is a closed ideal of $\mathcal K(Z)$.

Let $k \in B$ be arbitrary. If $k \in A \cap B$, then  
$\big( U\theta_A\big)_{0,k} = U_{0,k} \in \mathcal A(X_k,Y)$ by assumption.
On the other hand, if $k \in B \setminus A$, then $\big( U\theta_A\big)_{0,k} = 0$. 
We deduce that $U\theta_A  \in \mathcal I_B$ as  $\big( U\theta_A\big)_{0,0} = 0$.
Moreover, if $k \in A$ is arbitrary, then
\[
\big(U - U\theta_A\big)_{0,k} = U_{0,k} - U_{0,k} = 0.
\] 
We conclude  that $U - U\theta_A \in \mathcal I_A$, since also 
$\big(U - U\theta_A\big)_{0,0} = U_{0,0} \in \mathcal A(Y)$ .
 \end{proof}

The following technical result was used in the proof of Claim (i) of Example \ref{infty}. 
It is a vector-valued analogue of a well-known fact for scalar operator matrices.

\begin{lma}\label{sum}
Suppose that $X_k$ are Banach spaces for $k \in \mathbb N$, and let
\[
Z = \big( \oplus_{k=1}^\infty X_k\big)_p
\]
 be the corresponding direct $\ell^p$-sum,
where $1 < p < \infty$. Let $P_k$ be the natural projection onto $X_k \subset Z$ and 
$J_k: X_k \to Z$ be the natural inclusion for $k \in \mathbb N$. Assume that 
$T = (T_{r,s}) \in \mathcal K(Z)$ is a compact operator such that 
the components $T_{r,s} =: P_rTJ_s \in \mathcal A(X_s,X_r)$ for all $r, s \in \mathbb N$. Then
$T = (T_{r,s}) \in \mathcal A(Z)$ is an approximable operator.
\end{lma}

\begin{proof}
Let $Q_n = \sum_{k=1}^n J_kP_k$ for $k \in \mathbb N$, that is, 
$Q_nz = (z_1,\ldots,z_n,0,\ldots)$
for $z = (z_k) \in Z$, is the natural projection of $Z$ onto 
the closed linear subspace $\oplus_{k=1}^n X_k \subset Z$.

We know by definition that
$Q_n =  \sum_{k=1}^n J_kP_k  \to I_Z$ pointwise on $Z$ as $n \to \infty$. 
It follows that 
\[
\Vert T - Q_nT\Vert = \Vert (I-Q_n)T\Vert \to 0  \textrm{ as } n \to \infty,
\]
since $T$ is a compact operator on $Z$. Moreover, by standard duality $Q_n^*$
is the natural projection 
of $Z^* =  \big( \oplus_{k=1}^\infty X_k^* \big)_{p'}$ onto the closed subspace 
$\oplus_{k=1}^n X_k^* \subset Z^*$
for $n \in \mathbb N$, where $p' \in (1,\infty)$ is the dual exponent of $p$. 
Hence 
\[
\Vert T - TQ_n\Vert = \Vert T^* - Q_n^*T^*\Vert \to 0  \textrm{ as } n \to \infty,
\]
since also $Q_n^* \to I_{Z^*}$ pointwise on $Z^*$ as $n \to \infty$ and 
$T^* \in \mathcal K(Z^*)$. (Here we use that $1 < p' < \infty$.)

Observe next that  $Q_nTQ_n = \sum_{k, l=1}^n J_kP_kTJ_lP_l 
= \sum_{k, l=1}^n J_kT_{k,l}P_l  \in \mathcal A(Z)$ for all $n \in \mathbb N$ 
by our assumption on $T$. 
Deduce from the above facts that
\[
\Vert T - Q_nTQ_n\Vert
\le \Vert T - Q_nT\Vert  + \Vert Q_n\Vert \cdot \Vert T - TQ_n \Vert \to 0
\]
as $n \to \infty$, which yields that $T \in \mathcal A(Z)$.
\end{proof}

We also formulate a version of Example \ref{infty} for finite direct sums.
Let $N \ge 2$ and consider $Z_N = \big( \oplus_{k=0}^N X_k\big)_p$\ , where 
$X_0 = Y$ and $X_k = X$ for $k = 1,\ldots, N$, and the spaces $X$ and $Y$ 
are those of Example \ref{infty}. Put $[N] = \{1,\ldots,N\}$. 

\begin{ex}\label{fin}
Let $Z_N = \big( \oplus_{k=0}^N X_k\big)_p$ be as above for $N \ge 2$
and define $\mathcal I_A \subset \mathcal K(Z_N)$ 
as in Example \ref{infty} for subsets $A \subset [N]$.
Then the family 
\[
\{\mathcal I_A: \emptyset \neq A \varsubsetneq [N]\}
\]
contains $2^N - 2$ closed ideals of $\mathcal K(Z_N)$ that satisfy
$\mathcal A(Z_N) \varsubsetneq \mathcal I_A  \varsubsetneq \mathcal K(Z_N)$.
Moreover, this family of closed ideals have the reverse order structure of 
the power set $(\mathcal P([N]),\subset)$.
\end{ex}

\begin{proof}
The argument is a simpler variant of that of Example \ref{infty}.
In this case $U = (U_{r,s})_{r,s = 0}^N \in \mathcal K(Z_N)$ is a $(N+1) \times (N+1)$
operator matrix, and  there are no convergence issues with the operators
on $Z_N$  or in the verification of the ideal properties.
We leave the details to the interested reader.
\end{proof}

\begin{remark}\label{Lideals}
We note that in Example \ref{infty} the closed ideals
$\mathcal I_A$ of $\mathcal K(Z)$  are not  ideals of 
$\mathcal L(X)$ for any $\emptyset \neq A \varsubsetneq \mathbb N$.

In fact, fix $r \notin A \cup \{0\}$ and $s \in A$. By construction we may pick
$T \in \mathcal K(X_r,Y) \setminus \mathcal A(X_r,Y)$. Define
$U = (U_{k.l}) \in \mathcal I_A$ through $U_{0,r} = T$ and $U_{k.l} = 0$ for 
$(k,l) \neq (0,r)$, and the operator $V = (V_{k,l}) \in \mathcal L(Z)$ by 
$V_{r,s} = I_X$ and $V_{k,l} = 0$ for $(k,l) \neq (r,s)$. It follows that
\[
(UV)_{0,s} = U_{0,r}V_{r,s} = T \notin \mathcal A(X_r,Y),
\] 
so that $UV \notin \mathcal I_A$. 
\end{remark}

We also draw attention to a few  problems suggested by our results. 

\smallskip

\begin{qu}
Let $X \subset \ell^p$ be a closed subspace such that $\mathfrak{A}_X \neq \{0\}$. 
Is it always possible to find a non-trivial closed ideal 
\[
\mathcal A(X) \varsubsetneq \mathcal J \varsubsetneq  \mathcal K(X)\  ?
\]
By Example \ref{newideal} there is such a closed subspace 
$X \subset \ell^p$  for $4 < p < \infty$. Note that for $1 \le p < 2$ one has 
$\overline{\mathcal{QN}_p(X)} = \mathcal A(X)$ for any $X$, so  
new classes are needed.
\end{qu}

\smallskip

There remains combinations of $(p,q)$ for which 
examples of a strict inclusion  \eqref{nKA} does not appear to be known.  

\smallskip

\begin{qu}\label{Q2}
Let $1 \le p < 2 < q < \infty$. Are there closed subspaces 
$X \subset \ell^p$ and $Y \subset \ell^q$ such that 
$\mathcal A(X,Y) \varsubsetneq \mathcal K(X,Y)$ ?
Note that there are closed subspaces $X \subset \ell^p$, where 
$p \in [1, \infty)$ and $p \neq 2$,
and $Z \subset c_0$ for which
\[
\mathcal A(X,Z) \varsubsetneq \mathcal K(X,Z)\  \textrm{ and } \ 
\mathcal A(Z,X) \varsubsetneq \mathcal K(Z,X).
\]
Namely, by Facts \ref{ap1}.(ii) there
is  a closed subspace  $X \subset \ell^p$, for which 
there is an operator $T \in \mathcal K(X) \setminus \mathcal A(X)$.
By Terzio\u{g}lu's compact factorisation theorem \cite{T} there is a closed subspace 
$Z \subset c_0$ and
a factorisation $T = BA$, where $A \in \mathcal K(X,Z)$ and $B \in \mathcal K(Z,X)$. 
Here  $A$ and $B$ cannot be approximable operators.
\end{qu}

\smallskip

The examples of Theorem \ref{ideals}  point to 
a number of further questions. 

\begin{qus}
(i) Let $1 \le p < 2 < q < \infty$, and suppose that $X \subset \ell^p$ and $Y \subset \ell^q$
are closed subspaces such that  $\mathfrak{A}_X \neq \{0\}$ and  $\mathfrak{A}_Y \neq \{0\}$. 
How may one construct further closed ideals 
\[
\mathcal A(X \oplus Y) \varsubsetneq \mathcal J \varsubsetneq  \mathcal K(X \oplus Y)\  
\]
in addition to those contained in Theorem \ref{ideals} and Example \ref{nonmax}?
It is straightforward to check that if $\mathcal A(X,Y) \varsubsetneq M  \subset \mathcal K(X,Y)$ 
is a closed linear subspace such that 
\begin{equation*}
SA \in M \textrm{ and }  BS \in M  \textrm{ for all } S \in M, \ A \in \mathcal K(X), 
\ B \in \mathcal K(Y),
\end{equation*}
then 
\[
\mathcal K_M =: \left( \begin{array}{ccc}
\mathcal A(X) & \mathcal A(Y,X)\\
M & \mathcal A(Y)\\
\end{array} \right)
\]
defines a new non-trivial closed ideal of $\mathcal K(X\oplus Y)$. 
We do not have examples of such non-trivial ideal components 
$M  \subset \mathcal K(X,Y)$ (this is also related to Question \ref{Q2} above).

(ii) Is it possible to iterate the construction of Theorem \ref{ideals}
for finite direct sums $\oplus_{k=1}^n X_k$ with $n \ge 3$ ?
\end{qus}

\smallskip

Our concluding remarks point to some research in parallel directions. 

\begin{remarks}
(i) Let $\mathcal S$ be the class of strictly singular operators, which defines a closed 
Banach operator ideal contained in the class $\mathcal R$ of the inessential operators. 
By Proposition \ref{real}  the quotient algebra 
\[
\mathfrak{S}^{\mathcal K}_X =: \mathcal S(X)/\mathcal K(X)
\]
is a radical Banach algebra for both real and complex scalars.
By a simple modification of Proposition \ref{quotient} 
the closed ideals of $\mathfrak{S}^{\mathcal K}_X$ correspond to the closed ideals
$\mathcal J$ satisfying 
$\mathcal K(X) \subset \mathcal J \subset \mathcal S(X)$.  
Much more is known about such ideals for 
classical Banach spaces $X$. For instance, 
there is a continuum of closed ideals  
$\mathcal K(X)  \subset \mathcal J \subset  \mathcal S(X)$
in the following cases: 

\smallskip

(a) $X = \ell^p \oplus \ell^q$ for $1 < p < q < \infty$ and
$X = L^p$ for $1 < p < \infty$ and $p \neq 2$ by Schlumprecht and Zs\'ak \cite{SZ}, 

(b)  $X = L^1$ and $X = C(0,1)$  by Johnson, Pisier and Schechtman \cite{JPS}. 

\smallskip

\noindent  In addition, Tarbard \cite{Td} has constructed for each $k \ge 2$
a real Banach space $X_k$ having a Schauder basis such that 
$dim(\mathfrak{S}^{\mathcal K}_{X_k}) = k-1$.
A similar remark also applies to the 
radical quotient algebras $\mathcal S(X)/\mathcal A(X)$, since 
$\mathcal K(X) = \mathcal A(X)$ for these results.

This line of research concerns classical spaces that have the A.P.,
and we refer to the introductions of \cite{SZ}, \cite{JPS} and  \cite{BKL}
for many additional results and  further references  about 
closed ideals of  $\mathcal L(X)$.

\smallskip

(ii) Let $(\mathcal I,\vert \cdot \vert_{\mathcal I})$ be a Banach operator ideal 
such that $\mathcal I \subset \mathcal K$. 
Properties of the generalised compact-by-approximable algebras 
\[
\mathfrak{A}_X^{\mathcal I} =: \mathcal I(X)/\overline{\mathcal F(X)}^{\vert \cdot \vert_{\mathcal I}}
\]
are studied in \cite{Wi} for Banach spaces $X$. 
The generalised setting of $\mathfrak{A}_X^{\mathcal I}$ allows for new phenomena and for 
examples of different structure of closed ideals. 
\end{remarks}

\bigskip

\textbf{Acknowledgements.} This work is  part of the Ph.D.-thesis of Henrik Wirzenius under the supervision of the first author. We are grateful to 
Garth Dales for drawing our attention to the problems in  \cite{D13} concerning the 
compact-by-approximable algebra $\mathfrak{A}_X$ and for encouraging 
work in this direction. We also thank  Niels Laustsen for various discussions and suggestions, 
and  Jan Fourie for supplying reference \cite{F83}.

Henrik Wirzenius gratefully acknowledges the financial support of the 
Magnus Ehrnrooth Foundation.

\end{document}